\documentclass[11pt]{amsart}

\usepackage[active]{srcltx}

\usepackage[pdftex,left=2.65cm,right=2.65cm,top=2.7cm,bottom=2.7cm]{geometry}

\usepackage[colorlinks=true,urlcolor=blue,
citecolor=red,linkcolor=blue,linktocpage,pdfpagelabels,
bookmarksnumbered,bookmarksopen]{hyperref}
\usepackage{amsmath,amssymb}
\usepackage{amstext}
\usepackage{amscd}
\usepackage{enumerate}
\usepackage{epsfig}
\usepackage{color}
\usepackage{esint}
\usepackage[utf8]{inputenc}
\usepackage{mathtools}

\newtheorem{theorem}{Theorem}

\newtheorem{proposition}{Proposition}
\newtheorem{claim}{Claim}
\newtheorem{lemma}{Lemma}
\numberwithin{lemma}{section}
\newtheorem{corollary}{Corollary}
\theoremstyle{definition}
\newtheorem{remark}{Remark}

\numberwithin{equation}{section}

\newcommand{\N}{{\mathbb N}}
\newcommand{\R}{{\mathbb R}}

\newcommand{\dif}{\,\mathrm{d}}

\allowdisplaybreaks

\begin{document}

\title[Nonlinear Schr\"{o}dinger-Poisson systems]{On a class of nonlinear Schr\"{o}dinger-Poisson systems involving a nonradial charge density}
\author{Carlo Mercuri}
\address{Department of Mathematics, Swansea University, Fabian Way, Crymlyn Burrows, Skewen, Swansea, U.K. SA1~8EN} 
\email{c.mercuri@swansea.ac.uk}
\author{Teresa Megan Tyler}
\address{Department of Mathematics, Swansea University, Fabian Way, Crymlyn Burrows, Skewen, Swansea, U.K. SA1~8EN}
\email{807424@swansea.ac.uk}

\begin{abstract}
In the spirit of the classical work of P. H. Rabinowitz on nonlinear Schr\"{o}dinger equations, we prove existence of mountain-pass solutions and least energy solutions to the nonlinear Schr\"{o}dinger-Poisson system
\begin{equation}\nonumber
\left\{
\begin{array}{lll}
  - \Delta u+ u + \rho (x) \phi  u = |u|^{p-1} u, \qquad &x\in \R^3,  \\
  \,\,\, -\Delta \phi=\rho(x) u^2,\ &   x\in \R^3,
\end{array}
\right.
\end{equation}
under different assumptions on $\rho: \mathbb R^3\rightarrow \mathbb R_+$ at infinity. Our results cover the range $p\in(2,3)$
where the lack of compactness phenomena may be due to the combined effect of the invariance by translations of a `limiting problem' at infinity
and of the possible unboundedness of the Palais-Smale sequences. Moreover, we find necessary conditions for concentration at points to occur for solutions to the singularly perturbed problem 
\begin{equation}\nonumber
\left\{
\begin{array}{lll}
  - \epsilon^2\Delta u+ u + \rho (x) \phi  u = |u|^{p-1} u, \qquad &x\in \R^3,  \\
  \,\,\, -\Delta \phi=\rho(x) u^2,\ &   x\in \R^3,
\end{array}
\right.
\end{equation}
in various functional settings which are suitable for both variational and perturbation methods.  

Palais-Smale Sequences, Lack of Compactness.
\end{abstract}

\maketitle
\tableofcontents

\section{Introduction}
We study existence of positive solutions to the following nonlinear Schr\"{o}dinger-Poisson system

\begin{equation}\label{main SP system}
\left\{
\begin{array}{lll}
  - \Delta u+ u + \rho (x) \phi  u = |u|^{p-1} u, \qquad &x\in \R^3,  \\
  \,\,\, -\Delta \phi=\rho(x) u^2,\ &   x\in \R^3,
\end{array}
\right.
\end{equation}\\
\noindent with $p\in(2,5)$ and $\rho:\R^3 \to \R$ a nonnegative measurable function which represents a non-constant `charge' corrector to the density $u^2$. In the context of the so-called Density Functional Theory, variants of system (\ref{main SP system}) appear as mean field approximations of quantum many--body systems, see \cite{Bao Mauser Stimming}, \cite{Catto}, \cite{Lions Hartree Fock}. The positive Coulombic potential, $\phi$, represents a repulsive interaction between particles, whereas the local nonlinearity $|u|^{p-1} u$ generalises the $u^{5/3}$ term introduced by Slater \cite{Slater} as local approximation of the exchange potential in Hartree--Fock type models, see e.g. \cite{Bokanowski Lopez Soler}, \cite{Mauser}. \\

Within a min-max setting and in the spirit of Rabinowitz \cite{Rabinowitz}, we study existence and qualitative properties of the solutions to (\ref{main SP system}), highlighting those phenomena which are driven by $\rho.$ The system \eqref{main SP system} `interpolates' the classical equation

\begin{equation}\label{Kwong}
  - \Delta u+  u = u^{p} , \qquad x\in \R^3,  
 \end{equation}\\
whose positive solutions have been classified by Kwong \cite{Kwong}, with 

\begin{equation}\label{SP constant}
\left\{
\begin{array}{lll}
  - \Delta u+ u +  \phi  u = |u|^{p-1} u, \qquad &x\in \R^3,  \\
  \,\,\, -\Delta \phi= u^2, & x\in \R^3,
\end{array}
\right.
\end{equation}\\
\noindent studied by several authors in relation to existence, nonexistence, multiplicity and behaviour of the solutions in the semi-classical limit; see e.g. \cite{Ambrosetti},  \cite{Benci and Fortunato},  \cite{Catto}, and references therein. In the case $\rho(x)\to0$ as $|x|\to+\infty$, (\ref{Kwong}) has been exploited as limiting equation to tackle existence/compactness questions related to certain classes of systems similar to \eqref{main SP system}, see e.g. \cite{Cerami and Molle} and \cite{Cerami and Vaira}. In the present paper we consider instances where the convergence of approximating solutions to \eqref{main SP system} is not characterised by means of (\ref{Kwong}), namely the cases where, as $|x|\rightarrow +\infty$, it holds either that $\rho\rightarrow +\infty$ (`coercive case'), or that $\rho\rightarrow \rho_\infty>0$ (`non-coercive case'). The latter corresponds to the case where nontrivial solutions of \eqref{SP constant} (up to coefficients) cause lack of compactness phenomena to occur. The main difficulty in this context is that, despite the extensive literature, a full understanding of the set of positive solutions to \eqref{SP constant} has not yet been achieved (symmetry, non-degeneracy, etc.). \\

The autonomous system (\ref{SP constant}), as well as \eqref{main SP system}, presents various mathematical features which are not shared with nonlinear Schr\"odinger type equations, mostly related to lack of compactness phenomena. In a pioneering work \cite{RuizJFA}, radial functions and constrained minimisation techniques have been used, over a certain natural constraint manifold defined combining the Pohozaev and Nehari identities, yielding existence results of positive radial solutions to \eqref{SP constant} for all $p\in(2,5)$. Again in a radial setting, a variant of system \eqref{SP constant} has been studied more recently in \cite{Jeong and Seok}. When $p\leq 2$ the change in geometry of the associated energy functional causes differing phenomena to occur. In \cite{RuizJFA} existence, nonexistence and multiplicity results have been shown to be sensitive to a multiplicative factor for the Poisson term. Nonexistence results for \eqref{SP constant} have also been obtained in $\R^3$ in the range $p\geq 5$ (see e.g.\ \cite{D'Aprile and Mugnai}). In the presence of potentials, however, existence may occur when $p=5$, as it has been recently shown in \cite{Cerami and Molle arXiv}. Ambrosetti and Ruiz \cite{Ambrosetti and Ruiz} improved upon these early results by using the so-called `monotonicity trick' introduced by Struwe \cite{Struwe} and formulated in the context of the nonlinear Schr\"odinger equations by Jeanjean \cite{Jeanjean} and Jeanjean-Tanaka \cite{Jeanjean and Tanaka}, in order to show the existence of multiple bound state solutions to \eqref{SP constant}.\\

Related problems involving a non-constant charge density $\rho$, and in the presence of potentials, have been studied. The vast majority of works involve the range $p>3$ since, when $p\leq3$, one has to face two major obstacles in applying the minimax methods: constructing bounded Palais-Smale sequences and proving that the Palais-Smale condition holds, see e.g \cite{RuizJFA}, and \cite{Ambrosetti and Ruiz}, \cite{Mercuri}. Cerami and Molle \cite{Cerami and Molle} and  Cerami and Vaira \cite{Cerami and Vaira} studied the system

\begin{equation}\label{SP all potentials}
\left\{
\begin{array}{lll}
  - \Delta u+  V(x) u + \lambda \rho (x)  \phi u = K(x) |u|^{p-1} u, \qquad &x\in \R^3,  \\
  \,\,\, -\Delta \phi=\rho(x) u^2, &  x\in \R^3,
\end{array}
\right.
\end{equation}\\
\noindent where $\lambda>0$ and $V(x)$, $\rho(x)$ and $K(x)$ are nonnegative functions in $\R^3$ such that

\begin{equation}\label{assumptions rho to 0}
\lim_{|x|\to+\infty} \rho(x)=0, \,\, \lim_{|x|\to+\infty} V(x)=V_{\infty}>0,\,\, \lim_{|x|\to+\infty} K(x)=K_{\infty}>0,
\end{equation}\\
and, under suitable assumptions on the potentials, proved the existence of positive ground state and bound state solutions for $p\in (3,5)$. In \cite{Bonheure and Mercuri} existence of positive solutions to \eqref{SP all potentials} in the range $p\in[3,5)$ has been proved under suitable assumptions on the potentials that guarantee some compact embeddings of weighted Sobolev spaces into weighted $L^{p+1}$ spaces. 
Vaira \cite{Vaira} also studied system \eqref{SP all potentials}, in the case that

\begin{equation}\label{assumptions rho to constant}
\lim_{|x|\to+\infty} \rho(x)=\rho_{\infty}>0, \,\,  V(x)\equiv 1,\,\,  \lim_{|x|\to+\infty} K(x)=K_{\infty}>0, 
\end{equation}\\
and, assuming $\lambda>0$ and $K(x) \not\equiv 1$, proved the existence of positive ground state solutions for $p\in(3,5)$. In a recent and interesting paper, Sun, Wu and Feng (see Theorem 1.4 of \cite{Sun Wu Feng}) have shown the existence of a solution to \eqref{SP all potentials} for $p\in (1,3]$, assuming \eqref{assumptions rho to constant} and $K(x)\equiv1$, provided $\lambda$ is sufficiently small and $\int_{\R^3}\rho(x)\phi_{\rho,w_{\lambda}}w_{\lambda}^2 < \int_{\R^3}\rho_{\infty}\phi_{\rho_{\infty},w_{\lambda}}w_{\lambda}^2$, where $(w_{\lambda},\phi_{\rho_{\infty},w_{\lambda}})$ is a positive solution to

\begin{equation*}
\left\{
\begin{array}{lll}
  - \Delta u+   u + \lambda \rho_{\infty}  \phi u = |u|^{p-1} u, \qquad & x\in \R^3,  \\
  \,\,\, -\Delta \phi=\rho_{\infty} u^2,& x\in \R^3.\\
\end{array}
\right.
\end{equation*}\\
\noindent Their results are obtained using the fact that all nontrivial solutions to \eqref{SP all potentials} lie in a certain manifold $M_{\lambda}^-$ (see Lemma 6.1 in \cite{Sun Wu Feng}) to show that the energy functional $J_{\lambda}$ is bounded from below on the set of nontrivial solutions to \eqref{SP all potentials}. We believe that this is necessary to prove Corollary 4.3 in \cite{Sun Wu Feng}, and, ultimately, to prove Theorem 1.4 in \cite{Sun Wu Feng}, and thus the existence result is only viable in the reduced range $2.18 \approx \frac{-2+\sqrt{73}}{3}< p\leq 3$ and provided the additional assumption $ \frac{3p^2+4p-23}{2(5-p)} \rho(x)+ \frac{p-1}{2}(\nabla \rho (x),x) \geq 0$ also holds. In this range of $p$ and under these assumptions, as observed in \cite{Sun Wu Feng}, solutions are ground states. 

\subsection{Main results} In light of the above results, we aim to study existence and qualitative properties of the solutions to (\ref{main SP system}), in the various functional settings corresponding to different hypotheses on the behaviour of $\rho$ at infinity. Throughout the paper we set $D^{1,2}(\R^3)=D^{1,2}$ as the space defined as 

\[D^{1,2}(\R^3) \coloneqq \{ u\in L^6(\R^3) : \nabla u \in L^2(\R^3)\},\]\\
and equipped with norm

\[ ||u||_{D^{1,2}(\R^3)} \coloneqq ||\nabla u||_{L^2(\R^3)}.\]\\
\noindent It is well-known that if $u^2\rho \in L^1_{\textrm{loc}
}(\R^3)$ is such that 

\begin{equation}\label{double integral finite}
\int_{\R^3}\int_{\R^3}\frac{u^2(x)\rho (x) u^2(y) \rho (y)}{|x-y|}\dif x\dif y < +\infty,
\end{equation}\\
\noindent then,

\[\phi_u (x) \coloneqq \int_{\R^3}\frac{\rho (y)u^2(y)}{\omega |x-y|}\dif y \in D^{1,2}(\R^3)\]\\
\noindent is the unique weak solution in $D^{1,2}(\R^3)$ of the Poisson equation

\[-\Delta \phi = \rho (x) u^2\]\\
\noindent and it holds that 

\begin{equation}\label{double integral}
\int_{\R^3} |\nabla \phi_u|^2 = \int_{\R^3} \rho \phi_u u^2  \dif x =\int_{\R^3}\int_{\R^3}\frac{u^2(x)\rho (x) u^2(y) \rho (y)}{\omega|x-y|}\dif{x} \dif{y}.
\end{equation}\\
\noindent Here, $\omega =4\pi.$
\noindent Using the explicit representation of $\phi_u$ the system \eqref{main SP system} reduces to solving the problem

\begin{equation}\label{SP one equation}
 - \Delta u+  u + \rho (x) \phi_u  u = |u|^{p-1} u,\\
 \end{equation}\\
whose positive solutions are critical points of the functional 

\begin{equation}\label{definition I}
I(u) \coloneqq \frac{1}{2}\int_{\R^3}(|\nabla u|^2 + u^2)+\frac{1}{4}\int_{\R^3}  \rho\phi_u u^2 -\frac{1}{p+1}\int_{\R^3}u_+^{p+1},
\end{equation}\\
\noindent which is natural to define in $E(\R^3)\subseteq H^1(\R^3)$   \\
\[E(\R^3) \coloneqq \left\{u\in D^{1,2}(\R^3) \,: \, ||u||_{E} < +\infty\right\},\]
\noindent where

\[||u||_{E}^2 \coloneqq \int_{\R^3}(|\nabla u|^2 +  u^2)\dif x + \left(\int_{\R^3}\int_{\R^3}\frac{u^2(x)\rho (x) u^2(y) \rho (y)}{|x-y|}\dif x\dif y\right)^{1/2}.\]\\

Variants of this space have been studied since the work of P.L. Lions \cite{Lions Hartree}, see e.g. \cite{RuizARMA}, and \cite{Bellazzini},\cite{Bonheure and Mercuri}, \cite{Mercuri Moroz Van Schaftingen}.
\noindent We recall that by the classical Hardy-Littlewood-Sobolev inequality, it holds that 

\begin{equation}\label{HLS intro}
\left|\, \int_{\R^3}\int_{\R^3}\frac{u^2(x)\rho (x) u^2(y) \rho (y)}{\omega|x-y|}\dif{x} \dif{y}\, \right|\leq C ||\rho u^2||_{L^{\frac{6}{5}}(\R^3)}^2,
\end{equation}\\
\noindent for some $C>0$. Thus, if $u\in H^1(\R^3)$, we see that, depending on the assumptions on $\rho$, we may not be able to control the Coulomb integral \eqref{double integral} using the natural bound provided by HLS inequality. This may be the case if e.\ g.\ $\rho(x)\to+\infty$ as $|x|\to+\infty.$ For these reasons, in all of the results of the present paper we also analyse those instances where $E(\R^3)$ and $H^1(\R^3)$ do not coincide. \\
 

\indent As a warm-up observation we have the following theorem regarding existence in the coercive case for $p\geq3$.\\
 \begin{theorem}\label{theorem existence coercive0}[{\bf Coercive case: existence of mountain pass solution for $p\geq 3$}]
Suppose $\rho \in C(\R^3)$ is nonnegative and $\rho(x) \to +\infty$ as $|x| \to +\infty$.  Then, for any $p \in [3,5)$, there exists a solution, $(u, \phi_{u})\in E(\R^3) \times D^{1,2}(\R^3)$, of \eqref{main SP system}, whose components are positive functions. In particular, $u$ is a mountain pass critical point of $I$ at level $c$, where $c$ is the min-max level defined in \eqref{minmax level0}.
\end{theorem}

\begin{proof}
Since $\rho(x) \to +\infty$ as $|x| \to +\infty$, then $E$ is compactly embedded in $L^{p+1}(\R^3)$ by Lemma \ref{compact embedding} below, and therefore the existence of a Mountain Pass solution $u$ to \eqref{main SP system} is provided by Theorem 1 of \cite{Bonheure and Mercuri}. 
Both $u,\phi_u$ are positive by the strong maximum principle, and this concludes the proof.
\end{proof}

\bigskip

\indent It is also worth finding conditions such that the term $\rho u^2$ goes to zero at infinity, since the whole right hand side of the Poisson equation is classically interpreted as a `charge density'. This is provided by the following.\\

\begin{proposition}\label{charge decay}[{\bf Decay of $u$ and $\rho u^2$}]
Let $\rho: \R^3\rightarrow \R$ be continuous and nonnegative, $p\in[1,5],$ and $(u,\phi_u)\in E(\R^3)\times D^{1,2}(\R^3)$ be solution to (\ref{main SP system}).
Assume that $u$
is nonnegative. Then, for every $\gamma \in (0,1)$, there exists $C>0$ such that \\

$u(x)\leq C e^{-\gamma(1+|x|)}\qquad \qquad \textrm{($L^2$-decay)}.$ \\

 \noindent If, in addition, $\rho$ is such that \\

$(i)$ $\liminf_{|x|\rightarrow \infty}\rho(x)|x|^{1-2\alpha}>A$ \\

$(ii)$ $\limsup_{|x|\rightarrow \infty} \rho(x)e^{-\beta(1+|x|)^\alpha}\leq B$ \\

\noindent for some $\alpha,\beta, A,B>0$, with $\beta<2\sqrt A,$ 
then, for some constant $C>0$, it holds that \\

$(a)$ $u(x)\leq C e^{-\sqrt A(1+|x|)^\alpha}$ \\

\noindent and therefore\\


 
 $(b)$ $ \rho(x)u^2(x)=O(e^{(\beta-2\sqrt A)(1+|x|)^\alpha}), \qquad \textrm{as}\,\,|x|\rightarrow +\infty.$ \\ 
\end{proposition}
\begin{proof}
The conclusion easily follows by Theorem 6 in \cite{Bonheure and Mercuri} (see also \cite{Bonheure and Van Schaftingen}). More precisely, setting $W(x)\coloneqq 1+\frac{\rho(x)}{|x|}$, the $L^2$-decay follows as $W(x)\geq 1$ and therefore 

\[\liminf_{|x|\rightarrow +\infty}W(x)>\gamma^2\]\\
is automatically satisfied for every $\gamma\in(0,1)$. Moreover, note that by $(i)$ it follows that 

\[\liminf_{|x|\rightarrow +\infty}W(x)|x|^{2-2\alpha}\geq\liminf_{|x|\rightarrow +\infty}\rho(x)|x|^{1-2\alpha}>A\]\\
which yields $(a)$ again by Theorem 6 in \cite{Bonheure and Mercuri}. This concludes the proof.
\end{proof}
\bigskip
With these preliminary results in place, we first study the case of coercive $\rho$, namely $\rho(x)\to+\infty$ as $|x|\to+\infty$, and work in the natural setting for this problem, $E(\R^3)$. When $p\in (2,3)$, we make use of the aforementioned `monotonicity trick' exploiting the structure of our functional, in order to construct bounded Palais-Smale sequences for small perturbations of \eqref{main SP system}. We are able to prove that these sequences converge using a compact embedding established in Lemma \ref{compact embedding}. We finally show that these results extend to the original problem and obtain the following theorem.\\

\begin{theorem}\label{theorem existence coercive}[{\bf Coercive case: existence of mountain pass solution for $p\in (2,3)$}]
Suppose $\rho \in C(\R^3) \cap W^{1,1}_{\textrm{loc}
}(\R^3)$ is nonnegative and $\rho(x) \to +\infty$ as $|x| \to +\infty$. Suppose further that $k\rho(x)\leq (x, \nabla \rho)$ for some $k>\frac{-2(p-2)}{(p-1)}$. Then, for any $p \in (2,3)$, there exists a solution, $(u, \phi_{u})\in E(\R^3) \times D^{1,2}(\R^3)$, of \eqref{main SP system}, whose components are positive functions. In particular, $u$ is a mountain pass critical point of $I$ at level $c$, where $c$ is the min-max level defined in \eqref{minmax level0}.\\
\end{theorem}

After establishing these results, we prove the existence of least energy solutions for all $p\in(2,5)$. It is important to note that for $p\in(3,5)$ the solutions provided by the following corollary coincide with those provided by Theorem \ref{theorem existence coercive0}. For $p\in(2,3]$, we make use of a minimising sequence in order to obtain the result, however we do not know whether the least energy solutions provided by what follows are the same as those provided by Theorem \ref{theorem existence coercive0} ($p=3$) and Theorem \ref{theorem existence coercive}.\\

 \begin{corollary}\label{theorem existence groundstate coercive}[{\bf Coercive case: existence of a least energy solution for $p\in(2,5)$}]
Suppose $\rho \in C(\R^3)$ is nonnegative and $\rho(x) \to +\infty$ as $|x| \to +\infty$. If $p\in(2,3)$, suppose, in addition, that $\rho \in W^{1,1}_{\textrm{loc}
}(\R^3)$ and $k\rho(x)\leq (x, \nabla \rho)$ for some $k>\frac{-2(p-2)}{(p-1)}$. Then, for all $p \in (2,5)$, there exists a solution, $(u, \phi_{u})\in E(\R^3) \times D^{1,2}(\R^3)$, of \eqref{main SP system}, whose components are positive functions, such that $u$ is a least energy critical point of $I$. \\
\end{corollary}

\begin{remark} 
If we define 

\begin{equation}\label{def of F}
\mathcal I(u) \coloneqq \frac{1}{2}\int_{\R^3}(|\nabla u|^2 + u^2)+\frac{1}{4}\int_{\R^3}  \rho\phi_u u^2 -\frac{1}{p+1}\int_{\R^3}|u|^{p+1},
\end{equation}\\
then, under the same assumptions on $\rho$ as in Corollary \ref{theorem existence groundstate coercive}, we can prove the existence of a least energy critical point for $\mathcal I$ for all $p\in(2,5)$  by following similar techniques to those used in the proof of Corollary \ref{theorem existence groundstate coercive}. Since for $p>3$ the mountain pass level is equal to the infimum on the Nehari manifold, in this range it is possible to select a positive groundstate critical point for $\mathcal I.$ It is not clear whether this is also the case for $p\in (2,3].$  \\ 
\end{remark}
\medskip

We then focus on the case of non-coercive $\rho$, namely when $\rho(x) \to \rho_{\infty}>0$ as $|x|\to +\infty$. For this problem, $E(\R^3)$ coincides with the larger space $H^1(\R^3)$, and so we look for solutions $(u,\phi_u)\in H^1(\R^3) \times D^{1,2}(\R^3)$ of \eqref{main SP system}. Our method relies on an {\it a posteriori} compactness analysis of bounded Palais-Smale sequences (in the spirit of the classical book of M.\ Willem \cite{Willem}), in which we find that any possible lack of compactness is related to the invariance by translations of the subcritical `problem at infinity' associated to \eqref{SP one equation perturbed}, namely 

\begin{equation}\label{SP infinity}
 - \Delta u+   u +\rho_{\infty} \bar{\phi}_u  u = |u|^{p-1} u,
 \end{equation}\\
\noindent where $\bar{\phi}_u (x) \coloneqq \int_{\R^3}\frac{\rho_{\infty}u^2(y)}{\omega |x-y|} \dif{y} \in D^{1,2}(\R^3)$. Positive solutions of \eqref{SP infinity} are critical points of the corresponding functional, $I^{\infty}:H^1(\R^3) \to \R$, defined as

\begin{equation}\label{definition I infinity}
I^{\infty}(u) \coloneqq \frac{1}{2}\int_{\R^3}(|\nabla u|^2 +  u^2)+\frac{1}{4}\int_{\R^3}\rho_{\infty} \bar{\phi}_u u^2  -\frac{1}{p+1}\int_{\R^3}u_+^{p+1}.
\end{equation}\\
\noindent When $p\in(2,3)$, we define perturbations of $I$ and $I^{\infty}$, namely $I_{\mu}$ and $I^{\infty}_{\mu}$ (see Section \ref{mu levels}), as follows

\begin{equation*}
I_{\mu}(u) \coloneqq \frac{1}{2}\int_{\R^3}(|\nabla u|^2 +  u^2)+\frac{1}{4}\int_{\R^3} \rho \phi_u u^2 -\frac{\mu}{p+1}\int_{\R^3}u_+^{p+1},\\
\end{equation*}\\
and 
\begin{equation*}
I_{\mu}^{\infty}(u) \coloneqq \frac{1}{2}\int_{\R^3}(|\nabla u|^2 +  u^2)+\frac{1}{4}\int_{\R^3}\rho_{\infty} \bar{\phi}_u u^2  -\frac{\mu}{p+1}\int_{\R^3}u_+^{p+1}.
\end{equation*}\\

The aforementioned {\it a posteriori} compactness analysis is provided by the following proposition. There are several compactness results of similar flavour since the pioneering works of P.L. Lions \cite{Lions} and Benci-Cerami \cite{Benci Cerami}, which include more recent contributions in the context of Schr\"odinger-Poisson systems, see e.g.\ \cite{Cerami and Vaira}, \cite{Vaira}, \cite{Cerami and Molle arXiv}. We point out that these recent results are mostly in the range $p>3$, for Palais-Smale sequences constrained on Nehari manifolds, and for functionals without positive parts, unlike our result.

\begin{proposition} \label{splitting theorem} [{\bf Global compactness for bounded PS sequences}]
Suppose $\rho \in C(\R^3)$ is nonnegative and $\rho(x) \to \rho_{\infty}\geq0$ as $|x| \to +\infty$. Let  $p\in(2,5)$ and $\mu \in \left[\frac{1}{2},1\right]$ and assume $(u_n)_{n\in\N}\subset H^1(\R^3)$ is a bounded Palais-Smale sequence for $I_{\mu}$. Then, there exists $l\in \N$, a finite sequence $(v_0,\ldots, v_l)\subset H^1(\R^3)$, and $l$ sequences of points $(y_n^j)_{n\in\N}\subset \R^3$, $1\leq j \leq l$, satisfying, up to a subsequence of $(u_n)_{n\in\N}$,\\
\begin{enumerate}[(i)]
\item $v_0$ is a nonnegative solution of \eqref{SP one equation perturbed},\\
\item $v_j$ are nonnegative, and possibly nontrivial, solutions of \eqref{SP one equation perturbed infinity} for $1\leq j \leq l$,\\
\item $| y_n^j | \to +\infty$, $|y_n^j - y_n^{j'}| \to +\infty$ as $n\to+\infty$ if $j\neq j'$,\\
\item $||u_n-v_0- \sum_{j=1}^{l} v_j(\cdot -y_n^j)||_{H^1(\R^3)}\to 0$ as $n\to+\infty$,\\
\item $||u_n||_{H^1(\R^3)}^2 \to \sum_{j=0}^{l} ||v_j||_{H^1(\R^3)}^2$ as $n\to+\infty$,\\
\item $I_{\mu}(u_n) = I_{\mu}(v_0)+\sum_{j=1}^{l}I_{\mu}^{\infty}(v_j)+o(1)$.\\
\end{enumerate}
\end{proposition}
\begin{remark}
In the case $\rho_{\infty}=0$, the limiting equation \eqref{SP one equation perturbed infinity} reduces to coincide with $\eqref{Kwong}$.\\
\end{remark}
Roughly speaking, the presence of $u_+$ in the functional $I_{\mu}$ and Sobolev's inequality imply that $(u_n)_-\rightarrow 0$ in $L^{p+1}.$ It is therefore possible to use an observation to the classical Brezis-Lieb lemma \cite{Brezis and Lieb} made in \cite{Mercuri and Willem} to show that it also holds that $(u_n-v_0)_-\rightarrow 0$ in $L^{p+1}.$ In the proof we also take advantage of recent nonlocal versions of the Brezis-Lieb lemma, see \cite{Bellazzini Frank Visciglia} and \cite{Mercuri Moroz Van Schaftingen}.\\
\begin{remark}
Dropping the $+$ subscript in the definition of $I_{\mu}$ and simply observing that for every continuous path $\gamma : [0,1]\rightarrow H^1$ it holds that $I_{\mu}(\gamma(t))=I_{\mu}(|\gamma(t)|),$ a quantitative deformation argument (see e.g.\ Lemma 2.3 and Theorem 2.8 in \cite{Willem}) would allow us, for almost every $\mu \in \left[\frac{1}{2}, 1\right]$, to construct a bounded Palais-Smale sequence $(u_{\mu,n})_{n\in\N}$ for $I_\mu$ at the level $c_\mu$ (defined in \eqref{minmax level}) such that 

\[\textrm{dist}(u_{\mu,n},\mathcal P)\rightarrow 0, \qquad \mathcal P = \{u\in H^1(\R^3)\,|\, u_-\equiv 0\}.\]\\
We have opted for a less abstract approach to positivity.\\
\end{remark}

In the case $p\in(2,3)$, we use Proposition \ref{splitting theorem} together with Pohozaev's and Nehari's identities to show that a sequence of approximated critical points, constructed by means of the `monotonicity trick', is relatively compact. This enables us to obtain the following result.\\

\begin{theorem}\label{general theorem}[{\bf Non-coercive case: existence of mountain pass solution for $p\in(2,3)$}]
Suppose $\rho \in C(\R^3) \cap W^{1,1}_{\textrm{loc}
}(\R^3)$ is nonnegative, $\rho(x) \to \rho_{\infty}>0$ as $|x| \to +\infty$, and $k\rho(x)\leq (x, \nabla \rho)$ for some $k>\frac{-2(p-2)}{(p-1)}$. Suppose further that either\\

$(i)$ $c<c^{\infty}$ \\

\noindent or\\

$(ii)$ $\rho(x) \leq \rho_{\infty}$ for all $x\in \R^3$, with strict inequality, $\rho(x) < \rho_{\infty}$, on some ball $B\subset \R^3,$\\

\noindent where $c$ and (resp.) $c^\infty$ are min-max levels defined in (\ref{minmax level0}) and (resp.) (\ref{minmax level00}).
Then, for any $p \in (2,3)$, there exists a solution, $(u, \phi_{u})\in H^1(\R^3) \times D^{1,2}(\R^3)$, of \eqref{main SP system}, whose components are positive functions. In particular, $u$ is a mountain pass critical point of $I$ at level $c$. \\
\end{theorem}

The non-coercive case turns out to be more `regular' with respect to compactness issues when $p\geq 3$. In fact, we can show that the Palais-Smale condition holds at the mountain pass level $c$ and as a consequence we have the following\\
\begin{theorem}\label{noncoercive large p} [{\bf Non-coercive case: existence of mountain pass solution for $p\geq 3$}] Suppose $\rho\in C(\R^3)$ is nonnegative and $\rho(x)\rightarrow \rho_\infty>0$ as $|x|\rightarrow +\infty$. Let either of the
following conditions hold:\\

$(i)$ $c<c^{\infty}$ \\

\noindent or\\

$(ii)$ $\rho(x) \leq \rho_{\infty}$ for all $x\in \R^3$, with strict inequality, $\rho(x) < \rho_{\infty}$, on some ball $B\subset \R^3,$\\

\noindent where $c$ and (resp.) $c^\infty$ are minimax levels defined in (\ref{minmax level0}) and (resp.) (\ref{minmax level00}).
Then, for any $p \in [3,5)$ there exists a solution, $(u, \phi_{u})\in H^1(\R^3) \times D^{1,2}(\R^3)$, of \eqref{main SP system}, whose components are positive functions. In particular, $u$ is a mountain pass critical point of $I$ at level $c$.\\ 
\end{theorem}

We follow up the previous two theorems with a result giving the existence of least energy solutions in the non-coercive case. When $p\in(3,5)$ the existence follows relatively straightforwardly using the Nehari characterisation of the mountain pass level, and when $p\in(2,3]$ we use a minimising sequence together with Proposition \ref{splitting theorem} to obtain the result. \\

 \begin{corollary}\label{theorem existence groundstate noncoercive}[{\bf Non-coercive case: existence of least energy solution for $p\in (2,5)$}]
Suppose $\rho \in C(\R^3)$ is nonnegative, $\rho(x) \to \rho_{\infty}>0$ as $|x| \to +\infty$, and one of the following conditions hold:\\

$(i)$ $c<c^{\infty}$ \\

\noindent or\\

$(ii)$ $\rho(x) \leq \rho_{\infty}$ for all $x\in \R^3$, with strict inequality, $\rho(x) < \rho_{\infty}$, on some ball $B\subset \R^3,$\\

\noindent where $c$ and (resp.) $c^\infty$ are minimax levels defined in (\ref{minmax level0}) and (resp.) (\ref{minmax level00}). If $p\in(2,3)$, suppose in addition that $\rho \in W^{1,1}_{\textrm{loc}
}(\R^3)$ and $k\rho(x)\leq (x, \nabla \rho)$ for some $k>\frac{-2(p-2)}{(p-1)}$. Then, for all $p \in (2,5)$, there exists a solution, $(u, \phi_{u})\in H^1(\R^3) \times D^{1,2}(\R^3)$, of \eqref{main SP system}, whose components are positive functions, such that $u$ is a least energy critical point of $I$. \\
\end{corollary}

\begin{remark} 
By following similar techniques to those used in the proof of Corollary \ref{theorem existence groundstate noncoercive}, we can show that under the same assumptions as this corollary (with obvious modifications to the minimax levels), there exists a least energy solution for $\mathcal I$, defined in \eqref{def of F}, for all $p\in(2,5)$. As in the coercive case, it is not clear if we can select a positive groundstate for $p\in (2,3]$.\\ 
\end{remark}

After establishing these existence results, it is natural to ask if the non-locality of the equation allows us to find localised solutions. Moreover, we are interested in removing any compactness condition. For these reasons we focus on the equation

\begin{equation}\label{SPwithlambda}
\left\{
\begin{array}{lll}
  -\epsilon^2 \Delta u+ \lambda u + \rho (x) \phi  u = |u|^{p-1} u,  \quad &x\in \R^3 \\
  \,\,\, -\Delta \phi=\rho(x) u^2, &x\in \R^3,
\end{array}
\right.
\end{equation}\\
\noindent with $\rho:\R^3 \to \R$ a nonnegative measurable function, $\lambda\in \R$, and $\lambda > 0$, taking advantage of a shrinking parameter $\epsilon \sim \hbar\ll 1$ which behaves like the Planck constant in the so-called `semiclassical limit'. In this direction, Ianni and Vaira \cite{Ianni and Vaira} notably showed that concentration of semiclassical solutions to 

\begin{equation*}
\left\{
\begin{array}{lll}
  -\epsilon^2 \Delta u+V(x) u + \rho (x) \phi u = |u|^{p-1} u,   \quad &1< p < 5, \quad x\in \R^3 \\
  \,\,\, -\Delta \phi=\rho(x) u^2, &x\in \R^3,
\end{array}
\right.
\end{equation*} \\
occurs at stationary points of the external potential $V$ using a Lyapunov-Schmidt approach (in the spirit of the Ambrosetti-Malchiodi monograph \cite{Ambrosetti and Malchiodi} on perturbation methods), whereas in \cite{Bonheure Di Cosmo Mercuri} concentration results have been obtained using a variational/penalisation approach in the spirit of del Pino and Felmer \cite{del Pino and Felmer}. In particular, in \cite{Bonheure Di Cosmo Mercuri} the question of studying concentration phenomena which are purely driven by $\rho$ has been raised. None of the aforementioned contributions have dealt with necessary conditions for concentration at points in the case $V\equiv \textrm{constant}$ and in the presence of a variable charge density function $\rho$. We manage to fill this gap, by obtaining a necessary condition, related to $\rho$, for the concentration at points for solutions to \eqref{SPwithlambda} both in $H^1(\R^3)$ and in $E(\R^3),$ which are the suitable settings for the study of concentration phenomena with perturbative and variational techniques, respectively.\\

\begin{theorem} \label{necessary conditions E} [{\bf Necessary conditions in $E$}]
Suppose that $\rho\in C^1(\R^3)$ is nonnegative and $|\nabla \rho(x)|=O(|x|^ae^{b|x|})$ as $|x|\to+\infty$ for some $a>0$ and some $b\in\R$. Let $p\in [2,5)$ and let $(u_{\epsilon}, \phi_{u_{\epsilon}})\in E(\R^3) \times D^{1,2}(\R^3)$ be a sequence of positive solutions of \eqref{SPwithlambda}. Assume that $u_{\epsilon}$ concentrate at a point $x_0$ for sufficiently small $\epsilon$, meaning that $\forall \delta >0$, $\exists \epsilon_0 >0$, $\exists R>0$ such that $u_{\epsilon}(x) \leq \delta$ for $|x-x_0|\geq \epsilon R$, $\epsilon<\epsilon_0$.  Then, $\nabla \rho (x_0) =0$. \\
\end{theorem}

The above result is obtained in the spirit of \cite{Wang} using classical blow-up analysis, uniform decay estimates, and Pohozaev type identities.\\
\begin{remark} Since we deal with concentrating solutions, we use the mean value theorem to control the growth of $\rho$ with the assumption on $\nabla \rho$ in order to apply the dominated convergence theorem in the proof of the theorem (see Claim 5). We note that this assumption is not needed in Theorem \ref{necessary conditions H} as we work with a bounded $\rho$ and therefore the application of the dominated convergence theorem is more immediate. \end{remark}

\begin{remark} When $b>0$ the proof of Theorem \ref{necessary conditions E} Claim 5 is sensitive to $\epsilon$ being smaller than the ratio $\frac{\sqrt{\lambda}}{b}$. This ratio arises as the proof consists of balancing the aforementioned growth of $\rho$ and $\nabla \rho$ with the a priori exponential decay of the concentrating solutions in order to apply the dominated convergence theorem.\\ \end{remark}

\begin{theorem} \label{necessary conditions H}[{\bf Necessary conditions in $H^1$}]
Suppose that $\rho\in C^1(\R^3)$ is nonnegative and that $\rho, \nabla \rho$ are bounded. Let $p\in [2,5)$ and let $(u_{\epsilon}, \phi_{u_{\epsilon}})\in H^1(\R^3) \times D^{1,2}(\R^3)$ be a sequence of positive solutions of \eqref{SPwithlambda}. Assume that $u_{\epsilon}$ concentrate at a point $x_0$, meaning that $\forall \delta >0$, $\exists \epsilon_0 >0$, $\exists R>0$ such that $u_{\epsilon}(x) \leq \delta$ for $|x-x_0|\geq \epsilon R$, $\epsilon<\epsilon_0$.  Then, $\nabla \rho (x_0) =0$.\\
\end{theorem}

\begin{remark}It is possible to relax the global boundedness assumption on $\rho$ and/or on $\nabla \rho$ when working in $H^1(\R^3)$, if we make a growth assumption on $\nabla \rho$. Namely, if we work in $H^1(\R^3)$ and have adequate local integrability on $\rho$ to ensure $\int_{\R^3}\int_{\R^3} \frac{u^2(x)\rho(x)u^2(y)\rho(y)}{|x-y|}\dif x\, \dif y <+\infty$, typically identified using the Hardy-Littlewood-Sobolev inequality, the statement of the theorem and the proof is identical to that of Theorem \ref{necessary conditions E}.\end{remark}

\begin{remark} In the proof of both Theorem \ref{necessary conditions E} and Theorem \ref{necessary conditions H}, one actually finds the condition $\rho(x_0)\nabla\rho(x_0)=0.$ We believe that this may be a necessary condition in the case $\rho$ is allowed to change sign on a small region. \end{remark}

\subsection{Related questions} In our opinion, there are a number of interesting questions related to our work which are worth studying in future projects.\\

\noindent {\bf A. Radial versus non-radial solutions.} In the case $\rho$ is a radial function one can restrict on functions having the same symmetry to find radial solutions, using Palais criticality principle, in all the above scenarios (coercive/non-coercive cases, for low/large $p$). It is not clear how to compare the energy levels nor the symmetry of the solutions with those that one finds using the above non-radial approaches.  \\

\noindent {\bf B. Variational characterisation.} 
As mentioned above, when $p\in(2,3]$, it is not obvious whether the mountain pass critical points for $I,$ are least energy solutions. Namely, for $p\in(2,3]$, there is no clear relation between the solutions found in Theorem \ref{theorem existence coercive0}  (for $p=3$) and Theorem \ref{theorem existence coercive} with those found in Corollary \ref{theorem existence groundstate coercive}, as well as between the solutions found in Theorem \ref{noncoercive large p} (for $p=3$) and Theorem \ref{general theorem} with those found in Corollary \ref{theorem existence groundstate noncoercive}.\\
 
\noindent {\bf C. Multiplicity.} For $p>3$, we believe that the existence of infinitely many solutions can be proved following Ambrosetti-Ruiz paper \cite{Ambrosetti and Ruiz}, for instance in the case of (possibly non-radial) coercive $\rho.$ We suspect that the constrained minimisation approach in \cite{Sun Wu Feng} may help refining the approach in \cite{Ambrosetti and Ruiz}, to obtain a multiplicity result for $p\leq 3.$\\

\noindent {\bf D. `Sharp' necessary conditions for point concentration.} 
Is it possible to allow a faster growth for $\rho$ in the necessary conditions for point concentration? The proof we provide is based on the uniform exponential decay of solutions, which is essentially due to the $L^2$ setting. \\

\noindent {\bf E. Sufficient conditions for point concentration.} 
Following a personal communication of Denis Bonheure \cite{Denis} setting\\ $$I_\epsilon(u)= \frac{1}{2}\int_{\R^3}(\epsilon^2|\nabla u|^2 +  u^2)+\frac{1}{4}\int_{\R^3}  \rho\phi_u u^2 -\frac{1}{p+1}\int_{\R^3}|u|^{p+1}$$
and $$I_0(u)= \frac{1}{2}\int_{\R^3}(|\nabla u|^2 +  u^2) -\frac{1}{p+1}\int_{\R^3}|u|^{p+1},$$\\
taking $u\in C^{1}_c(\R^3)$ and using the scaling $$u_\epsilon= u\Big(\frac{\cdot-x_0}{\epsilon}\Big),$$ one finds the expansion\\ $$\frac{\epsilon^{-3}I_\epsilon(u_\epsilon)-I_0(u)}{\epsilon^2}\approx \rho^2(x_0) \int_{\R^3}\int_{\R^3} \frac{ u^2 (y) u^2 (x)}{4\pi|x-y|}\dif y \, \dif x, \qquad \epsilon \ll 1.$$\\
Inspired by this formal calculation,  in a forthcoming paper, we will prove concentration at strict local minima of $\rho^2$.

\subsection{Outline} The paper is organised as follows. In Section \ref{regularity}, we recall some properties of the space $E(\R^3)$, prove regularity and positivity results for solutions to the nonlinear Schr\"{o}dinger-Poisson system, and establish a useful Pohozaev identity for these solutions, the proof of which can be found in the appendix. In Section \ref{mu levels}, we outline the min-max setting and define the levels $c_{\mu}$ and $c_{\mu}^{\infty}$, $c$, and $c^{\infty}$, which are relevant for both the coercive and non-coercive cases. We then find lower bounds on the functions $I_{\mu}$ and $I_{\mu}^{\infty}$, when restricted to the set of nontrivial solutions which are fundamental in relation to compactness properties of Palais-Smale sequences.  In Section \ref{coercive section}, we study the case of a coercive $\rho$ and establish that this coercivity is a sufficient condition for the compactness of the embedding $E(\R^3)\hookrightarrow L^{p+1}(\R^3)$. This enables us, using the min-max setting of Section \ref{mu levels}, to prove existence of mountain pass solutions in the coercive case for $p\in (2,3)$ (Theorem \ref{theorem existence coercive0}). We then use a minimisation argument to prove the existence of least energy solutions (Corollary \ref{theorem existence groundstate coercive}). In Section \ref{non-coercive section}, we focus on a non-coercive $\rho$ and we first establish a representation result for bounded Palais-Smale sequence for $I_{\mu}$ in Proposition \ref{splitting theorem}. Using the min-max setting of Section \ref{mu levels} and the lower bounds found in this section, we prove existence of mountain pass solutions for $p\in (2,3)$ (Theorem \ref{general theorem}). We then show that for $p\geq3$ the Palais-Smale condition holds for $I$ at the level $c$, following which the proof of Theorem \ref{noncoercive large p} easily follows. We finally prove the existence of least energy solutions in the non-coercive case for $p\in (2,5)$ (Corollary \ref{theorem existence groundstate noncoercive}). In Section \ref{Necessary conditions for concentration of semiclassical states}, we obtain necessary conditions for the concentration at points in both $E$ (Theorem \ref{necessary conditions E}) and $H^1$ (Theorem \ref{necessary conditions H}).

\subsection{Notation} 
We use the following notation throughout:
\begin{itemize}
\item $L^p(\Omega)$, with $\Omega \subseteq \R^3$ and $p \geq 1$, is the usual Lebesgue space. $L^p(\R^3)=L^p$.
\item The H\"{o}lder space $C^{k,\alpha}(\Omega)$, with $\Omega \subseteq \R^3$ and $\alpha \in (0,1]$, is the set of functions on $\Omega$ that are $k$-fold differentiable and whose $k$-fold derivatives are H\"{o}lder continuous of order $\alpha$.
\item $H^1,$ $W^{m,p}$ are classical Sobolev spaces.
\item $H^{-1}(\R^3)=H^{-1}$ denotes the dual space of $H^1(\R^3)$.
\item $D^{1,2}(\R^3)=D^{1,2}$ is the space defined as 
\[D^{1,2}(\R^3) \coloneqq \{ u\in L^6(\R^3) : \nabla u \in L^2(\R^3)\},\]
and equipped with norm
\[ ||u||_{D^{1,2}(\R^3)} \coloneqq ||\nabla u||_{L^2(\R^3)}.\]
\item $E(\R^3)=E$ is the space defined as\\
\[E(\R^3) \coloneqq \left\{u\in D^{1,2}(\R^3) \,: \, ||u||_{E} < +\infty\right\},\]
where
\[||u||_{E}^2 \coloneqq \int_{\R^3}(|\nabla u|^2 +  u^2)\dif x + \left(\int_{\R^3}\int_{\R^3}\frac{u^2(x)\rho (x) u^2(y) \rho (y)}{|x-y|}\dif x\dif y\right)^{1/2}.\]
\item We set
\[\phi_u (x) \coloneqq \int_{\R^3}\frac{\rho (y)u^2(y)}{4\pi |x-y|}\dif y,\]
and
\[\bar\phi_u (x) \coloneqq \int_{\R^3}\frac{\rho_{\infty}u^2(y)}{4\pi |x-y|}\dif y.\]
\item For any $\eta>0$ and any $z\in\R^3$, $B_{\eta}(z)$ is the ball of radius $\eta$ centered at $z$. For any $\eta>0$, $B_{\eta}$ is the ball of radius $\eta$ centered at $0$. 

\item $S_{p+1}\coloneqq \inf_{u\in H^1(\R^3)\setminus \{0\}} \frac{||u||^2_{H^1(\R^3)}}{||u||^2_{L^{p+1}(\R^3)}}$ is the best Sobolev constant for the embedding of $H^1(\R^3)$ into $L^{p+1}(\R^3)$.
\item Let $A\subset \R^3$. Then, we define $$\chi_{A}(x)\coloneqq \left\{
\begin{array}{lll}
  1, \qquad &x\in A,  \\
  0, & x\not\in A.
\end{array}
\right.
$$
\item $C,C_1, C'$, etc., denote generic positive constants.
\item Asymptotic Notation: For real valued functions $f(t),g(t)\geq0$, we write:
\begin{itemize}
\item $f(t)\lesssim g(t)$ if there exists $C>0$ independent of $t$ such that $f(t)\leq Cg(t)$. 
\item  $f(t)=o(g(t))$ as $t\to+\infty$ if and only if $g(t)\neq0$ and $\lim_{t\to+\infty} \frac{f(t)}{g(t)}=0$.
\item $f(t)=O(g(t))$ as $t\to+\infty$ if and only if there exists $C_1>0$ such that $f(t)\leq C_1 g(t)$ for $t$ large.
\end{itemize}

\end{itemize}

\section*{Acknowledgements} C.M. would like to thank Antonio Ambrosetti and David Ruiz for having drawn his attention to questions related to the Schr\"odinger-Poisson systems. The same author would like to thank Michel Willem for inspiring discussions on questions related to the Palais-Smale condition. Both authors would like to thank Norihisa Ikoma for his suggestion to consider a minimising sequence in order to prove the existence of least energy solutions. Last, but not least, both authors would like to thank an anonymous referee for the valuable and constructive comments.

\section{Preliminaries} \label{regularity}
\subsection{The space $E(\R^3)$} 
Let us assume that $\rho$ is continuous and nonnegative. It is easy to see that $E(\R^3)$ is a uniformly convex Banach space. As a consequence it is reflexive and, in particular, the unit ball is weakly compact. Reasoning as in Proposition 2.4 in \cite{RuizARMA} and Proposition 2.10 in \cite{Mercuri Moroz Van Schaftingen} a sequence $(u_n)_{n\in \mathbb N}\subset E(\R^3)$ is weakly convergent to $u\in E$ if and only if is bounded and converges in $L^1_{\textrm{loc}
}(\R^3).$ In particular, $\phi_{u_n}\rightharpoonup \phi_u$ in $D^{1,2}(\R^3).$ The following nonlocal Brezis-Lieb lemma is very useful to study the compactness of Palais-Smale sequences. \\

\begin{lemma}[\cite{Bellazzini Frank Visciglia}, \cite{Mercuri Moroz Van Schaftingen}] \label{nonlocalBL}[{\bf Nonlocal Brezis-Lieb lemma}]
Let $(u_n)_{n\in \mathbb N}\subset E(\R^3)$ be a bounded sequence such that $u_n \rightarrow u$ almost everywhere in $\R^3.$ Then it holds that 

$$\lim_{n\rightarrow \infty} \Big[\|\nabla\phi_{u_n}\|^2_{L^2(\R^3)}-\|\nabla \phi_{u_n-u}\|^2_{L^2(\R^3)}\Big]=\| \nabla \phi_{u} \|^2_{L^2(\R^3)}.$$
\end{lemma}


\subsection{Regularity and positivity}\label{Regularity and positivity}
\begin{proposition}\label{reg}[{\bf Regularity and positivity}] Let $p\in[1,5],$ $\rho\in L^\infty_{\textrm{loc}}(\R^3)\setminus\{0\}$ be nonnegative and $(u,\phi_u)\in E(\R^3)\times D^{1,2}(\R^3)$ be a weak solution of the problem

\begin{equation}\label{poh system}
\left\{
\begin{array}{lll}
  - \Delta u+b u +c  \rho (x) \phi u = d|u|^{p-1} u, \qquad &x\in \R^3,  \\
  \,\,\, -\Delta \phi=\rho(x) u^2,\qquad &x\in \R^3,
\end{array}
\right.
\end{equation}\\
\noindent with $b, c, d \in \R_+$.
Assume that $u$
is nonnegative. Then, $u,\phi_u\in W^{2,q}_{\textrm{loc}}(\R^3),$ for every $q\geq1$, and so $u,\phi_u \in C^{1,\alpha}_{\textrm{loc}}(\R^3).$
If, in addition, $u\not\equiv 0$, then $u,\phi_u>0$ everywhere.

\end{proposition}
\begin{proof}
Under the hypotheses of the proposition, both $u$ and $\phi_u$ have weak second derivatives in $L^q_{\textrm{loc}
}$ for all $q<\infty$. In fact, note that from the first equation in \eqref{poh system}, we have that $-\Delta u = g(x,u)$, where 

\begin{align*}
|g(x,u)|&=|(-b u -c  \rho (x) \phi u + d|u|^{p-1} u| \\
&\leq C(1+|\rho \phi_u |+|u|^{p-1})(1+|u|)\\
&\coloneqq h(x)(1+|u|).\\
\end{align*}
Using our assumptions on $\rho$, $\phi_u$, and $u$, we can show that $h\in L^{3/2}_{\textrm{loc}
}(\R^3)$, which implies that $u\in L^q_{\textrm{loc}
}(\R^3)$ for all $q<+\infty$ (see e.g.\ \cite[p.\ 270]{Struwe Book}). Note that since $u^2\rho \in L^q_{\textrm{loc}
}(\R^3)$ for all $q<+\infty$, then by the second equation in \eqref{poh system} and the Calder\'{o}n-Zygmund estimates, we have that $\phi_u\in W^{2,q}_{\textrm{loc}
}(\R^3)$  (see e.g.\ \cite{Gilbarg and Trudinger}). This then enables us to show that $g\in L^q_{\textrm{loc}
}(\R^3)$ for all $q<+\infty$, which implies, by Calder\'{o}n-Zygmund estimates, that $u\in W^{2,q}_{\textrm{loc}
}(\R^3)$ (see e.g.\ \cite{Gilbarg and Trudinger}). The $C^{1,\alpha}_{\textrm{loc}}(\R^3)$ regularity of both $u,\phi_u$ is a consequence of Morrey's embedding theorem. Finally, the strict positivity
is a consequence of the strong maximum principle, and this concludes the proof.
\end{proof}

\begin{remark} If, in addition, $\rho\in C^{0,\alpha}_{\textrm{loc}
}(\R^3)$, then, by Schauder's estimates on both equations, it holds that $u,\phi_u\in C^{2,\alpha}_{\textrm{loc}
}(\R^N)$.
\end{remark}

\subsection{Pohozaev identity}

We can now establish a useful Pohozaev type identity for solutions to the nonlinear Schr\"{o}dinger-Poisson system that will be used on numerous occasions. Although these kind of identities are standard, since we do not find a precise reference, we give a proof in the appendix for the reader convenience.\\

\begin{lemma}\label{pohozaevlemma}[{\bf Pohozaev identity}]
Assume $\rho \in L^\infty_{\textrm{loc}
}(\R^3) \cap W^{1,1}_{\textrm{loc}
}(\R^3)$ is nonnegative and $p\in[1,5]$. Let $(u, \phi_u) \in E(\R^3) \times D^{1,2}(\R^3)$ be a weak solution of the problem \eqref{poh system}. Then, it holds that 

\begin{equation*}
\left|\frac{c}{2}\int_{\R^3}  \phi_u u^2 (x, \nabla \rho) \dif x\right| <+\infty \footnote{In the case $(x,\nabla \rho)$ changes sign, we set $$\int_{\R^3}  \phi_u u^2 (x, \nabla \rho) \dif x=\lim_{n\rightarrow \infty}\int_{B_{R_n}}  \phi_u u^2 (x, \nabla \rho) \dif x$$ for a suitable sequence of radii $R_n\rightarrow \infty.$ As part of the proof we can select a sequence $(R_n)_{n\in \mathbb N}$ such that this limit exists and is finite.},\\
\end{equation*}\\
\noindent and

\begin{equation}\label{pohozaev}
\frac{1}{2}\int_{\R^3} |\nabla u|^2 +\frac{3b}{2}\int_{\R^3}  u^2 +\frac{5c}{4}\int_{\R^3} \rho \phi_u u^2  + \frac{c}{2}\int_{\R^3} \phi_u u^2 (x, \nabla \rho) -\frac{3d}{p+1}\int_{\R^3} |u|^{p+1}=0.\\
\end{equation}\\
\end{lemma}



\section{The min-max setting: definition of $c_\mu$, $c_\mu^\infty$, $c$, and $c^{\infty}$} \label{mu levels}

In what is to come, we will first examine the existence of solutions of \eqref{main SP system} in the case of a coercive potential $\rho$ (see Section \ref{coercive section}). The appropriate setting for this problem will be the space $E(\R^3)\subset H^1(\R^3)$. We begin by recalling that solving \eqref{main SP system} reduces to solving \eqref{SP one equation} with $\phi_u (x) \coloneqq \int_{\R^3}\frac{u^2(y)\rho (y)}{\omega |x-y|}dy \in D^{1,2}(\R^3)$. It will also be useful to introduce a perturbation of \eqref{SP one equation}, namely

\begin{equation}\label{SP one equation perturbed}
 - \Delta u+  u +  \rho (x) \phi_u u = \mu |u|^{p-1} u,\qquad \mu \in \left[\frac{1}{2},1\right],\\
 \end{equation}\\
\noindent and to note that the positive solutions of this perturbed problem are critical points of the corresponding functional $I_{\mu}:E(\R^3) \to \R$, defined as

\begin{equation}\label{perturbed functional}
I_{\mu}(u) \coloneqq \frac{1}{2}\int_{\R^3}(|\nabla u|^2 +  u^2)+\frac{1}{4}\int_{\R^3} \rho \phi_u u^2 -\frac{\mu}{p+1}\int_{\R^3}u_+^{p+1}, \qquad \mu\in\left[\frac{1}{2},1\right].\\
\end{equation}\\
\noindent We will now show that $I_{\mu}$ has the mountain pass geometry in $E$ for each $\mu\in [\frac{1}{2},1]$.\\

\begin{lemma} \label{claim1}[{\bf Mountain-Pass Geometry for $I_\mu$}]
Suppose $\rho\in C(\R^3)$ is nonnegative and $p\in(2,5]$. Then, for each $\mu\in [\frac{1}{2},1]$, it holds:
\begin{enumerate}[(i)]
\item $I_{\mu}(0)=0$ and there exists constants $r,a >0$ such that $I_{\mu}(u)\geq a$ if $||u||_E=r$.
\item There exists $v\in E$ with $||v||_E>r$, such that $I_{\mu}(v) \leq 0$.\\
\end{enumerate}
\end{lemma}

\begin{proof}
We follow Lemma 14 in \cite{Bonheure and Mercuri}. To prove $(i)$ note that since $H^1(\R^3) \hookrightarrow L^{p+1}(\R^3)$ then for some constant $C>0$, it holds that

\begin{align*}
I_{\mu} (u) \geq \frac{1}{2}||u||_{H^1}^2+\frac{1}{4}\int_{\R^3}\rho \phi_u u^2  -C\mu||u||_{H^1}^{p+1}\\
\end{align*}
\noindent Now, from the definition of the norm in $E$ we can see that $4\pi\int_{\R^3} \rho \phi_u u^2 = \left(||u||^2_E- ||u||_{H^1}^2\right)^2$. Therefore, we have that

\begin{align*}
I_{\mu} (u) &\geq \frac{1}{2}||u||_{H^1}^2+\frac{1}{16\pi}\left(||u||^2_E- ||u||_{H^1}^2\right)^2 -C\mu||u||_{H^1}^{p+1}\\
&=\frac{1}{2}||u||_{H^1}^2+\frac{1}{4\pi}\left(\frac{1}{4}||u||^4_E- \frac{1}{2}||u||^2_E ||u||_{H^1}^2+\frac{1}{4} ||u||_{H^1}^4\right) -C\mu||u||_{H^1}^{p+1}.\\
\end{align*}
\noindent For some $\alpha\neq 0,$ using the elementary inequality 

\[\frac{1}{2}||u||^2_E ||u||_{H^1}^2\leq \frac{\alpha^2}{4}||u||^4_{H^1}+\frac{1}{4\alpha^2} ||u||_{E}^4\]\\
\noindent we have

\begin{align}\label{localmin}
I_{\mu} (u) &\geq \frac{1}{2}||u||_{H^1}^2+\frac{1}{4\pi}\left(\frac{1}{4}||u||^4_E- \frac{\alpha^2}{4}||u||^4_{H^1}-\frac{1}{4\alpha^2} ||u||_{E}^4+\frac{1}{4} ||u||_{H^1}^4\right) -C\mu||u||_{H^1}^{p+1}\nonumber \\
&= \frac{1}{2}||u||_{H^1}^2- \frac{1}{4\pi}\left(\frac{\alpha^2-1}{4}\right)||u||^4_{H^1}+\frac{1}{4\pi}\left(\frac{\alpha^2-1}{4\alpha^2}\right) ||u||_{E}^4 -C\mu||u||_{H^1}^{p+1}.\\
\nonumber
\end{align}
\noindent We now assume $||u||_E<\delta$ for some $\delta >0$, which also implies that $||u||_{H^1}^2<\delta^2$, and we take $\alpha>1$. Then, from \eqref{localmin}, we see that

\begin{align*}
I_{\mu} (u) &\geq \left[ \frac{1}{2}-\frac{1}{4\pi}\left( \frac{\alpha^2-1}{4} \right) \delta^2 -C \mu \delta^{p-1}\right] ||u||_{H^1}^2 +\frac{1}{4\pi}\left(\frac{\alpha^2-1}{4\alpha^2}\right) ||u||_{E}^4\\
&\geq \frac{1}{4\pi}\left(\frac{\alpha^2-1}{4\alpha^2}\right) ||u||_{E}^4, \quad \text{for }\delta \text{ sufficiently small.}\\
\end{align*}
\noindent Hence, we have shown that the origin is a strict local minimum for $I_{\mu}$ in $E$ if $p\in[2,5]$. 

To show $(ii)$, pick $u\in C^{1}(\R^3)$, supported in the unit ball, $B_1$. Setting $v_t(x) \coloneqq t^2u(tx)$ we find that 

\begin{equation}\label{rescaledI}
I_{\mu} (v_t) =\frac{t^3}{2} \int_{\R^3} |\nabla u|^2 + \frac{t}{2} \int_{\R^3} u^2 + \frac{t^3}{4} \int_{\R^3}\int_{\R^3} \frac{ u^2(y)\rho(\frac{y}{t})u^2(x)\rho(\frac{x}{t})}{\omega|x-y|}\dif y\,\dif x - \frac{ \mu t^{2p-1}}{p+1} \int_{\R^3}u_+^{p+1}.
\end{equation}\\
Since the Poisson term is uniformly bounded, namely for $t>1$ 

\begin{equation*}
\int_{\R^3}\int_{\R^3} \frac{ u^2(y)\rho(\frac{y}{t})u^2(x)\rho(\frac{x}{t})}{\omega|x-y|}\dif y\,\dif x 
\leq ||\rho||_{L^{\infty}(B_1)}^2 \int_{\R^3}\int_{\R^3} \frac{ u^2(y)u^2(x)}{\omega|x-y|}\dif y\,\dif x <+\infty,\
\end{equation*}\\
\noindent the fact that  $2p-1>3$ in \eqref{rescaledI} yields $I_{\mu} (v_{t}) \to -\infty$ as $t\to +\infty,$ and this is enough to prove $(ii)$. This concludes the proof.
\end{proof}

\bigskip

\noindent The previous lemma, as well as the monotonicity of $I_{\mu}$ with respect to $\mu$, imply that there exists $\bar{v}\in E\setminus \{0\}$ such that

\[I_{\mu} (\bar{v})\leq I_{\frac{1}{2}}(\bar{v}) \leq 0, \qquad \forall \mu \in \left[\frac{1}{2},1\right].\]\\
\noindent Thus, we can define, in the spirit of Ambrosetti-Rabinowitz \cite{Ambrosetti and Rabinowitz}, the min-max level associated with $I_{\mu}$ as

\begin{equation}\label{minmax level}
c_{\mu} \coloneqq \inf_{\gamma \in \Gamma} \max_{t\in [0,1]}  I_{\mu}(\gamma(t)),\\ 
\end{equation}
\noindent where $\Gamma$ is the family of paths

\[\Gamma \coloneqq \{ \gamma \in C([0,1], E) : \gamma (0) =0, \, \gamma(1)=\bar{v}\}.\]\\
\noindent It is worth emphasising that to apply the monotonicity trick \cite{Jeanjean} and \cite{Jeanjean and Tanaka} it is essential that the above class $\Gamma$ does not depend on $\mu.$\\

\begin{lemma} \label{claim3}
Suppose $\rho\in C(\R^3)$ is nonnegative and $p\in(2,5)$. Then:\\
\begin{enumerate}[(i)]
\item The mapping $\left[\frac{1}{2},1\right] \ni \mu \mapsto c_{\mu}$ is non-increasing and left-continuous.\\
\item For almost every $\mu \in [\frac{1}{2},1]$, there exists a bounded Palais-Smale sequence for $I_{\mu}$ at the level $c_{\mu}$. That is, there exists a bounded sequence $(u_n)_{n\in\N} \subset E$ such that $I_{\mu}(u_n) \to c_\mu$ and $I'_{\mu}(u_n)\to 0$.\\
\end{enumerate}
\end{lemma}

\begin{proof}
The proof of $(i)$ follows from Lemma 2.2 in \cite{Ambrosetti and Ruiz}. To prove $(ii)$, we notice that by Lemma \ref{claim1}, it holds that 

\[c_{\mu} = \inf_{\gamma \in \Gamma} \max_{t\in [0,1]}  I_{\mu}(\gamma(t)) >0 \geq \max\{I_{\mu}(0), I_{\mu}(\bar{v})\}, \qquad \forall \mu \in \left[\frac{1}{2},1\right].\]\\
\noindent Thus, the result follows by Theorem  1.1 in \cite{Jeanjean}.
\end{proof}
\medskip

\noindent With the previous result in place, we can define the set

\begin{equation}\label{set of mu}
\mathcal{M}\coloneqq \left\{\mu \in \left[\frac{1}{2},1\right] \, :\, \exists \text{ bounded Palais-Smale sequence for }I_{\mu}\text{ at the level }c_{\mu}\right\}.\\
\end{equation} \\
Since $I$ has the mountain pass geometry by Lemma \ref{claim1}, using $(i)$ of Lemma \ref{claim3}, we can now define the min-max level associated with $I$ as 

\begin{equation}\label{minmax level0}
c\coloneqq \left\{
\begin{array}{lll}
  c_1, \qquad &p\in(2,3),  \\
   \inf_{\gamma \in \bar{\Gamma}} \max_{t\in [0,1]}  I(\gamma(t)),\ &  p\in[3,5),
\end{array}
\right.
\end{equation}\\
where $\bar{\Gamma}$ is the family of paths 

\[\bar{\Gamma} \coloneqq \{ \gamma \in C([0,1], E(\R^3)) : \gamma (0) =0, \, I(\gamma(1))<0\}.\]\\
This finalises the preliminary min-max scheme for the case of a coercive $\rho$.\\

In Section \ref{non-coercive section}, we will then focus on the case of non-coercive $\rho$, namely $\rho(x)\to \rho_{\infty}$ as $|x|\to +\infty$, and the appropriate setting for this problem will be the space $H^1(\R^3)$. It will once again be useful to introduce a perturbation of \eqref{SP one equation}, namely, \eqref{SP one equation perturbed}, and to recall that the positive solutions of this perturbed problem are critical points of the corresponding functional, $I_{\mu}:H^1(\R^3) \to \R$, defined in \eqref{perturbed functional}. We note that Lemma \ref{claim1} and Lemma \ref{claim3} hold with $E(\R^3)=H^1(\R^3)$, and thus $\mathcal{M}$ can be defined as in \eqref{set of mu}. We now introduce the problem at infinity related to \eqref{SP one equation perturbed} in this case, namely 

\begin{equation}\label{SP one equation perturbed infinity}
 - \Delta u+  u +\rho_{\infty} \bar{\phi}_u  u = \mu |u|^{p-1} u,\qquad \mu\in\left[\frac{1}{2},1\right],\\
 \end{equation}\\
\noindent where $\bar{\phi}_u (x) \coloneqq \int_{\R^3}\frac{\rho_{\infty}u^2(y)}{\omega |x-y|}dy \in D^{1,2}(\R^3)$. Positive solutions of \eqref{SP one equation perturbed infinity} are critical points of the corresponding functional, $I_{\mu}^{\infty}:H^1(\R^3) \to \R$, defined as

\begin{equation}\label{perturbed functional infinity}
I_{\mu}^{\infty}(u) \coloneqq \frac{1}{2}\int_{\R^3}(|\nabla u|^2 +  u^2)+\frac{1}{4}\int_{\R^3}\rho_{\infty} \bar{\phi}_u u^2  -\frac{\mu}{p+1}\int_{\R^3}u_+^{p+1}, \qquad \mu\in\left[\frac{1}{2},1\right].
\end{equation}\\
\noindent It can be shown that $I_{\mu}^{\infty}$ satisfies the geometric conditions of the mountain-pass theorem, using similar arguments as those used in the proof of Lemma \ref{claim1}. We therefore define the min-max level associated with $I_{\mu}^{\infty}$ as

\begin{equation}\label{minmax level infinity}
c_{\mu}^{\infty}\coloneqq \inf_{\gamma \in \Gamma^{\infty}} \max_{t\in [0,1]}  I_{\mu}^{\infty}(\gamma(t)),
\end{equation}
where 

\[\Gamma^{\infty} \coloneqq \{ \gamma \in C([0,1], H^1(\R^3)) : \gamma (0) =0, \, I_{\frac{1}{2}}^{\infty}(\gamma(1))<0\}.\]\\
\noindent Moreover, we define the min-max level associated with $I^{\infty}$ as 

\begin{equation}\label{minmax level00}
c^{\infty}\coloneqq \left\{
\begin{array}{lll}
   c^{\infty}_{1}, \qquad &p\in(2,3),  \\
   \inf_{\gamma \in \bar{\Gamma}^{\infty}} \max_{t\in [0,1]}  I^{\infty}(\gamma(t)),\ &  p\in[3,5),
\end{array}
\right.
\end{equation}\\
where $\bar{\Gamma}^{\infty}$ is the family of paths

\[\bar{\Gamma}^{\infty} \coloneqq \{ \gamma \in C([0,1], E(\R^3)) : \gamma (0) =0, \, I^{\infty}(\gamma(1))<0\}.\]

\subsection{Lower bounds for $I$ and $I^{\infty}$}
In the next two lemmas, we establish lower bounds on $I_{\mu}$ and $I_{\mu}^{\infty}$, when restricted to nonnegative and nontrivial solutions of \eqref{SP one equation perturbed} and \eqref{SP one equation perturbed infinity}, respectively. These bounds will be used on numerous occasions.\\

\begin{lemma} \label{lower bound energy nontriv sols}
Suppose $\rho \in C(\R^3)$ is nonnegative and $\mu\in [\frac{1}{2},1]$. Define $\mathcal{A}\coloneqq\{u\in H^1(\R^3)\setminus \{0\} : u \text{ is a nonnegative solution to } \eqref{SP one equation perturbed}\}$. Then,  if $p\in[3,5)$, it holds that

\[\inf_{u\in \mathcal{A}} I_{\mu}(u)\geq \frac{p-1}{2(p+1)}\left(S_{p+1}\right)^{\frac{p+1}{p-1}}>0.\]\\
If $p\in(2,3)$, suppose, in addition, $ \rho\in W^{1,1}_{loc}(\R^3)$ and $k\rho(x)\leq (x, \nabla \rho)$ for some $k>\frac{-2(p-2)}{(p-1)}$. 
Then, it holds that

\[\inf_{u\in \mathcal{A}} I_{\mu}(u) \geq  C(k,p),\]\\
with $$C(k,p):=\left(\frac{2(p-2)+k(p-1)}{(3+2k)(p+1)}\right) \left(S_{p+1}\right)^{\frac{p+1}{p-1}} >0.$$\\
\end{lemma}

\begin{proof}
Let $\bar{u}\in H^1(\R^3)\setminus \{0\}$ be an arbitrary nonnegative solution of \eqref{SP one equation perturbed} such that $I_{\mu}(\bar{u})=\bar{c}$. Using the Sobolev embedding theorem and the fact that $I_{\mu}'(\bar{u})(\bar{u})=0$, we see that

\[S_{p+1}||\bar{u}||_{L^{p+1}}^2 \leq ||\bar{u}||_{H^{1}}^2 \leq ||\bar{u}||_{H^{1}}^2 + \int_{\R^3} \rho\phi_{\bar{u}} \bar{u}^2 = \mu ||\bar{u}||_{L^{p+1}}^{p+1}.\]\\
\noindent Since $\mu\leq 1$ it follows that

\begin{equation}\label{lower bound nontriv sols}
\left(S_{p+1}\right)^{\frac{p+1}{p-1}}\leq ||\bar{u}||_{H^{1}}^2.
\end{equation}\\
\noindent If $p\in[3,5)$, using the definition of $\bar{c}$ and Nehari's condition, we can see that 

$$\Big(\frac{1}{2}-\frac{1}{p+1}\Big)||\bar{u}||_{H^{1}}^2 \leq\bar{c},$$\\
and so the bound on $\bar{c}$ immediately follows from (\ref{lower bound nontriv sols}). If $p\in(2,3)$, we first note that since $I_{\mu}(\bar{u})=\bar{c}$, $I_{\mu}'(\bar{u})(\bar{u})=0$, and $\bar{u}=(\bar{u})_+$, then $\bar{u}$ satisfies 

\begin{equation}\label{SOE1}
\frac{1}{2}\int_{\R^3}(|\nabla \bar{u}|^2 + \bar{u}^2)+\frac{1}{4}\int_{\R^3} \rho \phi_{\bar{u}} \bar{u}^2 -\frac{\mu}{p+1}\int_{\R^3}\bar{u}^{p+1} =\bar{c},
\end{equation}\\
\noindent and\\
\begin{equation}\label{SOE2}
\int_{\R^3}(|\nabla \bar{u}|^2 + \bar{u}^2)+\int_{\R^3}\rho \phi_{\bar{u}} \bar{u}^2  -\mu\int_{\R^3}\bar{u}^{p+1} =0.
\end{equation}\\
\noindent Moreover, since $\bar{u}$ solves \eqref{SP one equation perturbed} then, by Lemma \ref{pohozaevlemma}, $\bar{u}$ must also satisfy the Pohozaev equality:

\begin{equation*}
\frac{1}{2}\int_{\R^3} |\nabla \bar{u}|^2 +\frac{3}{2}\int_{\R^3}   \bar{u}^2 +\frac{5}{4}\int_{\R^3}\rho \phi_{\bar{u}} \bar{u}^2  + \frac{1}{2}\int_{\R^3}  \phi_{\bar{u}}\bar{u}^2(x,\nabla \rho) -\frac{3\mu}{p+1}\int_{\R^3} \bar{u}^{p+1}=0.\\
\end{equation*}\\
\noindent We now recall that, by assumption, $k\rho(x)\leq (x, \nabla \rho)$ for some $k>\frac{-2(p-2)}{(p-1)}$. Using this in the above equality, we see that

\begin{equation}\label{SOE3}
\frac{1}{2}\int_{\R^3} (|\nabla \bar{u}|^2+ \bar{u}^2) +\left(\frac{5+2k}{4}\right)\int_{\R^3}\rho \phi_{\bar{u}} \bar{u}^2  -\frac{3\mu}{p+1}\int_{\R^3} \bar{u}^{p+1}\leq 0.
\end{equation}\\
\noindent For ease, we now set $\alpha=||\bar{u}||_{H^{1}}^2$, $\gamma=\int_{\R^3} \rho \phi_{\bar{u}} \bar{u}^2 $, and $\delta=\mu \int_{\R^3}\bar{u}^{p+1}$. From \eqref{SOE1}, \eqref{SOE2}, and \eqref{SOE3}, we can see that $\alpha$, $\gamma$, and $\delta$ satisfy

\begin{equation*}
\left\{
\begin{array}{ccccccc}
  \frac{1}{2}\alpha&+&\frac{1}{4}\gamma &-&\frac{1}{p+1}\delta &=&\bar{c},   \\
  \alpha&+&\gamma &-&\delta &=&0, \\
  \frac{1}{2}\alpha &+&\left(\frac{5+2k}{4}\right)\gamma &-&\frac{3}{p+1}\delta &\leq&0,
\end{array}
\right.
\end{equation*}\\
\noindent and so, it holds that

\[\delta \leq \frac{\bar{c} (3+2k)(p+1)}{2(p-2)+k(p-1)},\]\\
and

\[\alpha = \delta -\gamma.\]\\
Since $\gamma$ is nonnegative, we find 

\[\alpha \leq \alpha + \gamma =\delta \leq \frac{\bar{c} (3+2k)(p+1)}{2(p-2)+k(p-1)}.\]\\
\noindent This and (\ref{lower bound nontriv sols}) implies the statement, since $k>\frac{-2(p-2)}{(p-1)}>\frac{-3}{2}$ for $p\in(2,3).$ This concludes the proof.
\end{proof}

\medskip

\begin{lemma} \label{lower bound nontriv sols infinity}
If $p\in(2,5)$, $\mu \in [\frac{1}{2},1]$ and $u\in H^1(\R^3) \setminus \{0\}$ is a nonnegative solution of \eqref{SP one equation perturbed infinity}, then, it holds that

 \[I_{\mu}^{\infty}(u) \geq c_{\mu}^{\infty}>0.\] \\
Moreover, if $p\in(2,5)$ and $u\in H^1(\R^3) \setminus \{0\}$ is a nonnegative solution of \eqref{SP infinity}, then

 \[I^{\infty}(u) \geq c^{\infty}>0.\] \\
In both cases, $u>0.$
\end{lemma}

\begin{proof}
The lower bounds follow easily by similar arguments to those used in the proof of Proposition 3.4 in \cite{Ianni and Ruiz}. Since $u$ is nonnegative and nontrivial, then it is strictly positive by the strong maximum principle, and this concludes the proof. 
\end{proof}

\section{Existence: the case of coercive $\rho(x)$}\label{coercive section}

In this section we will examine the existence of solutions of \eqref{main SP system} in the case of a coercive potential $\rho$, namely $\rho(x)\to+\infty$ as $|x|\to+\infty$. In the following lemma, we establish that this coercivity is indeed a sufficient condition for the compactness of the embedding $E\hookrightarrow L^{p+1}(\R^3)$.\\

\begin{lemma}\label{compact embedding}
Assume $\rho(x) \to +\infty$ as $|x| \to +\infty$. Then, $E$ is compactly embedded in $L^{p+1}(\R^3)$ for all $p \in(1,5)$.\\
\end{lemma}

\begin{proof}We first recall that for any $u\in E$, it holds that 

\[ -\Delta \phi_u =\rho u^2,\]\\
where $\phi_u (x) \coloneqq \int_{\R^3} \frac{\rho(y)u^2(y)}{\omega|x-y|}dy\in \mathcal D^{1,2}(\R^3)$.
Testing this equation with $u_+$ and $u_-$ and using the Cauchy-Schwarz inequality, it follows that

\begin{align*}
\int_{\R^3} \rho|u|^3 &= \int_{\R^3}\nabla|u|\nabla \phi_u\\
&\leq \left(\int_{\R^3}|\nabla|u||^2\right)^{\frac{1}{2}}\left(\int_{\R^3}|\nabla\phi_u|^2\right)^{\frac{1}{2}}\\
&= \left(\int_{\R^3}|\nabla u|^2\right)^{\frac{1}{2}}\left(\int_{\R^3}\int_{\R^3} \frac{u^2(x)\rho(x) u^2(y) \rho(y)}{4\pi|x-y|}\right)^{\frac{1}{2}}\\
&\leq \left(\frac{1}{4\pi}\right)^{\frac{1}{2}}||u||^3_E.\\
\end{align*}
\noindent Thus, if $\rho>0$, this implies the continuous embedding $E \hookrightarrow L^3_{\rho}(\R^3)$, where $L^3_{\rho}(\R^3) \coloneqq \{u: \rho^{\frac{1}{3}}u \in L^3(\R^3)\}$, equipped with norm $||u||_{L^3_{\rho}}\coloneqq || \rho^{\frac{1}{3}}u||_{L^3}$. \\

Without loss of generality, assume $u_n\rightharpoonup 0$ in E. Since $\rho(x) \to +\infty$ as $|x| \to +\infty$, then for any $\epsilon >0$, there exists an $R>0$ such that

\begin{equation}\label{L3ballcomplement}
 \int_{\R^3\setminus B_R} |u_n|^3 =  \int_{\R^3\setminus B_R}\frac{\rho}{\rho} |u_n|^3<\epsilon  \int_{\R^3 \setminus B_R}\rho |u_n|^3 < \epsilon C,\\
\end{equation}\\
for some $C>0$. This and the classical Rellich theorem implies that, passing if necessary to a subsequence,

\begin{equation}\label{unto0}
 \int_{\R^3} |u_n|^3\to 0.\\
\end{equation}\\
\noindent Therefore, we have proved the lemma for $p=2$. Now, if $p\in(1,2)$, then, by interpolation, for some $\alpha \in (0,1)$, it holds that 

\begin{equation*}
||u_n||_{L^{p+1}(\R^3)} \leq ||u_n||_{L^{2}(\R^3)}^{\alpha}||u_n||_{L^{3}(\R^3)}^{1-\alpha} \to 0,\\
\end{equation*}\\
\noindent as the $L^{2}(\R^3)$ norm is bounded. The case $p\in(2,5)$ is similar using Sobolev's inequality, and this concludes the proof. 
\end{proof}

\subsection{Proof of Theorem \ref{theorem existence coercive}} We are now in position to prove the existence of mountain pass solutions.

\begin{proof}[Proof of Theorem \ref{theorem existence coercive}]
We first note that by Lemma \ref{claim3}, the set $\mathcal{M}$, defined in \eqref{set of mu}, is nonempty.\\

\begin{claim} \label{claim4}
The values $c_{\mu}$ are critical levels of $I_{\mu}$ for all $\mu \in \mathcal{M}$. Namely, there exists ${u} \in E$ such that $I_{\mu}({u}) = c_{\mu}$ and $I_{\mu}'({u})= 0$.
\end{claim}

By definition, for each $\mu \in \mathcal{M}$, there exists a bounded sequence $(u_n)_{n\in\N}\subset E$ such that $I_{\mu}(u_n) \to c_\mu$ and $I'_{\mu}(u_n)\to 0$. Since $(u_n)_{n\in\N}$ is bounded, there exists $u \in E$ such that, up to a subsequence, $u_n \rightharpoonup u$ in $E$. Using this and the fact that $E$ is compactly embedded in $L^{p+1}(\R^3)$ by Lemma \ref{compact embedding}, arguing as in Lemma 16 in \cite{Bonheure and Mercuri}, with $V(x)=1$ and $K(x)=\mu$, we see that for all $\delta>0$, there exists a ball $B\subset \R^3$ such that

\begin{equation}\label{BM1}
\limsup\limits_{n\to +\infty} \int_{\R^3 \setminus B} \rho \phi_{u_n} u_n^2 <\delta,
\end{equation}\\
and

\begin{equation}\label{BM2}
\limsup\limits_{n\to +\infty} \left| \int_{\R^3 \setminus B}\rho  \phi_{u_n} u_n u \right| <\delta.
\end{equation}\\
\noindent We then note that since $(u_n)_{n\in\N}$ is bounded in $E$, we also have that, up to a subsequence, $u_n \rightharpoonup u$ in $H^1$. Now, using this and the fact that $(u_n)_{n\in\N}$ is a bounded Palais Smale sequence for $I_{\mu}$, as well as \eqref{BM1}, \eqref{BM2}, and Lemma \ref{compact embedding}, we can reason as in Lemma 18 in \cite{Bonheure and Mercuri}, with $V(x)=1$ and $K(x)=\mu$, to see that 

\begin{equation}\label{H1 convergence}
\int_{\R^3}(|\nabla u_n|^2+ u_n^2) \to \int_{\R^3}(|\nabla u|^2+ u^2).
 \end{equation}\\
\noindent Thus, using \eqref{BM1} and the boundedness of $(u_n)_{n\in\N}$, we can argue as in the proof of Theorem 1 in \cite{Bonheure and Mercuri}, to see that

\begin{equation}\label{Coloumb convergence}
\int_{\R^3}\rho \phi_{u_n}u_n^2  \to \int_{\R^3}  \rho\phi_{u}u^2,
\end{equation}\\
\noindent which, when combined with \eqref{H1 convergence} and Lemma \ref{compact embedding}, implies that

\[I_{\mu}(u_n) \to I_{\mu}(u).\]\\
\noindent Therefore, we have shown

\[I_{\mu}(u)=c_{\mu}.\]\\
\noindent Moreover, by standard arguments, using the weak convergence $u_n \rightharpoonup u$ in $E$, we can show 

\[I'_{\mu}(u) =0.\]\\
\noindent We finally note that, by putting \eqref{H1 convergence} and \eqref{Coloumb convergence} together, we have that $||u_n||^2_{E} \to ||u||^2 _{E}$, and so by Lemma \ref{nonlocalBL}, it follows that $u_n \to u$ in $E$. This concludes the proof of Claim \ref{claim4}. \\

\begin{claim} \label{claim5}
Let $(\mu_n)_{n\in\N}$ be an increasing sequence in $\mathcal{M}$ such that $\mu_n \to 1$ and assume $(u_n)_{n\in\N} \subset E$ is such that $I_{\mu_n}(u_n)=c_{\mu_n}$ and $I'_{\mu_n}(u_n)=0$ for each $n$. Then, there exists ${u} \in E$ such that, up to a subsequence, $u_n \to {u}$ in $E$, $I(u)=c$, and $I'(u)=0$. 
\end{claim}

We first note that testing the equation $I'_{\mu_n}(u_n)=0$ with $(u_n)_-$, one sees that $u_n\geq0$ for each $n$. Therefore, it holds that $u_n$ satisfies 

\begin{equation}\label{mu equation}
-\Delta u_n+  u_n + \rho (x) \phi_{u_n}  u_n = \mu_n u_n^{p},
\end{equation}

\begin{equation}\label{coerciveSOE1}
\frac{1}{2}\int_{\R^3}(|\nabla u_n|^2 + u_n^2)+\frac{1}{4}\int_{\R^3} \rho \phi_{u_n} u_n^2 -\frac{\mu_n}{p+1}\int_{\R^3}u_n^{p+1} =c_{\mu_n},
\end{equation}\\
\noindent and 

\begin{equation}\label{coerciveSOE2}
\int_{\R^3}(|\nabla u_n|^2 + u_n^2)+\int_{\R^3}\rho \phi_{u_n} u_n^2  -\mu_n\int_{\R^3}u_n^{p+1} =0.
\end{equation}\\
\noindent Moreover, since $u_n$ solves \eqref{mu equation} then, using Lemma \ref{pohozaevlemma} and the assumption $k\rho(x)\leq (x, \nabla \rho)$ for some $k>\frac{-2(p-2)}{(p-1)}$, and arguing as in Lemma \ref{lower bound energy nontriv sols}, we see that

\begin{equation}\label{coerciveSOE3}
\frac{1}{2}\int_{\R^3} (|\nabla u_n|^2+ u_n^2) +\left(\frac{5+2k}{4}\right)\int_{\R^3}\rho \phi_{u_n} u_n^2  -\frac{3\mu_n}{p+1}\int_{\R^3} u_n^{p+1}\leq 0.
\end{equation}\\
\noindent Setting $\alpha_n=\int_{\R^3} (|\nabla u_n|^2+  u_n^2)$, $\gamma_n=\int_{\R^3} \rho \phi_{u_n} u_n^2 $, and $\delta_n=\mu_n \int_{\R^3} u_n^{p+1}$, we can see, from \eqref{coerciveSOE1}, \eqref{coerciveSOE2}, and \eqref{coerciveSOE3}, that $\alpha_n$, $\gamma_n$, and $\delta_n$ satisfy

\begin{equation}\label{mainSOE}
\left\{
\begin{array}{ccccccc}
  \frac{1}{2}\alpha_n&+&\frac{1}{4}\gamma_n &-&\frac{1}{p+1}\delta_n &=&c_{\mu_n},   \\
  \alpha_n&+&\gamma_n &-&\delta_n &=&0, \\
  \frac{1}{2}\alpha_n &+&\left(\frac{5+2k}{4}\right)\gamma_n &-&\frac{3}{p+1}\delta_n &\leq&0.
\end{array}
\right.
\end{equation}\\

\noindent Solving the system, we find that

\[\delta_n \leq \frac{c_{\mu_n} (3+2k)(p+1)}{2(p-2)+k(p-1)},\]\\
\[\gamma_n \leq \frac{-2c_{\mu_n}(p-5)}{2(p-2)+k(p-1)},\]\\
and

\[\alpha_n = \delta_n -\gamma_n.\]\\
\noindent Since $c_{\mu_n}$ is bounded, $k>\frac{-2(p-2)}{(p-1)}>\frac{-3}{2}$, and $\delta_n$, $\gamma_n$, and $\alpha_n$ are all nonnegative, we can deduce that $\delta_n$, $\gamma_n$, and $\alpha_n$ are all bounded. Hence, the sequence $(u_n)_{n\in\N}$ is bounded in $E$ and so there exists ${u} \in E$ such that, up to a subsequence, $u_n \rightharpoonup {u}$ in $E$. 

We now follow a similar procedure to that of Claim 1. Using the facts that $I'_{\mu_n}(u_n)= 0$, $u_n$ is bounded in $E$, $E$ is compactly embedded in $L^{p+1}(\R^3)$ by Lemma \ref{compact embedding}, and $\mu_n \to 1$, by an easy argument similar to the proof of Lemma 16 in \cite{Bonheure and Mercuri}, with $V(x)=1$ and $K(x)=\mu_n$, we have that for all $\delta>0$, there exists a ball $B\subset \R^3$ such that

\begin{equation}\label{BM1.1}
\limsup\limits_{n\to +\infty} \int_{\R^3 \setminus B} \rho\phi_{u_n} u_n^2  <\delta,
\end{equation}\\
and
\begin{equation}\label{BM2.1}
\limsup\limits_{n\to +\infty} \left| \int_{\R^3 \setminus B}\rho \phi_{u_n} u_n u  \right| <\delta.
\end{equation}\\
\noindent Now, using the facts that $I'_{\mu_n}(u_n)=0$ and $\mu_n \to 1$, as well as \eqref{BM1.1}, \eqref{BM2.1}, and Lemma \ref{compact embedding}, we can adapt the proof of Lemma 18 in \cite{Bonheure and Mercuri}, with $V(x)=1$ and $K(x)=\mu_n$, to see that 

\begin{equation}\label{H1 convergence 1}
\int_{\R^3}(|\nabla u_n|^2+ u_n^2) \to \int_{\R^3}(|\nabla u|^2+ u^2).
 \end{equation}\\
\noindent Finally, using \eqref{BM1.1}, \eqref{H1 convergence 1}, the boundedness of $u_n$, Lemma \ref{compact embedding}, and the fact that $\mu_n \to 1$, we can easily adapt the proof of Theorem 1 in \cite{Bonheure and Mercuri}, to see that

\begin{equation}\label{Coloumb convergence 1}
\int_{\R^3} \rho \phi_{u_n}u_n^2  \to \int_{\R^3} \rho \phi_{u}u^2 ,
\end{equation}

\begin{equation}\label{finalconvergence}
c_{\mu_n}=I_{\mu_n}(u_n) \to I(u),
\end{equation}\\
\noindent and
\[0=I'_{\mu_n}(u_n) \to I'(u).\]\\

\noindent As in Claim 1, we see that \eqref{H1 convergence 1} and \eqref{Coloumb convergence 1} imply that $||u_n||^2_{E} \to ||u||^2 _{E}$, and so by Lemma \ref{nonlocalBL}, it follows that $u_n \to u$ in $E$. We now recall that, for $p\in(2,3)$, it holds that $c_{\mu_n} \to c$ as $ \mu_n \nearrow 1$ by definition \eqref{minmax level0}. Thus, from \eqref{finalconvergence} it follows that $I(u)=c$. \\

\noindent 
{\bf Conclusion.} \noindent
Let $(\mu_n)_{n\in\N}$ be an increasing sequence in $\mathcal{M}$ such that $\mu_n \to 1$. By Claim \ref{claim4}, we can choose $(u_n)_{n\in\N} \subset E$ such that $I_{\mu_n}(u_n)=c_{\mu_n}$ and $I'_{\mu_n}(u_n)=0$ for each $n$. By Claim \ref{claim5}, it follows that that there exists ${u} \in E$ such that, up to a subsequence, $u_n \to {u}$ in $E$, $I(u)=c$, and $I'(u)=0$. Namely, we have shown $(u, \phi_{u})\in E(\R^3) \times D^{1,2}(\R^3)$ is a solution of \eqref{main SP system}. By the strong maximum principle $\phi_u$ is strictly positive. Testing the equation $I'(u)=0$ with $u_-$ one sees that $u\geq 0$ and, in fact, strictly positive as a consequence of the strong maximum principle. This concludes the proof.
\end{proof}

\medskip

\subsection{Proof of Corollary \ref{theorem existence groundstate coercive}} In the next proof, we show the existence of least energy solutions.

\begin{proof}[Proof of Corollary \ref{theorem existence groundstate coercive}]
When $p>3$ it is standard to see that the Mountain Pass level $c$ has the following characterisation  

\begin{equation}\label{Nehari}
c=\inf_{u\in \mathcal N} I(u), \qquad \mathcal N=\{u\in E\setminus \{0\}\,\, |\,\, I'(u)u=0\},
\end{equation}\\
see e.g. Theorem 5 in \cite{Bonheure and Mercuri}. It follows that the mountain pass solution $u$ found in Theorem \ref{theorem existence coercive0} is a least energy solutions of $I$ in this case. If $p\in(2,3]$, define 

\[c^*\coloneqq \inf_{u\in \mathcal{A}} I(u),\text{ where }\mathcal{A}\coloneqq\{u\in E(\R^3)\setminus \{0\} : u \text{ is a nonnegative solution to } \eqref{SP one equation}\}.\]\\
When $p=3$, we notice that the mountain pass critical point, $u$, that we found in Theorem \ref{theorem existence coercive0} is such that $u\in\mathcal{A}$. Similarly, when $p\in(2,3)$, the mountain pass critical point that we found in Theorem \ref{theorem existence coercive} is in $\mathcal{A}$.  Therefore, in both cases, $\mathcal{A}$ is nonempty and $c^*$ is well-defined.  Now, let $(w_n)_{n\in\N}\subset \mathcal A$ be a minimising sequence for $I$ on $\mathcal{A},$ namely $I(w_n)\to c^*$ as $n\to+\infty$ and $I'(w_n)=0$. If $p=3$, it holds that\\ $$c+1\geq (p+1)I(w_n)-I'(w_n)w_n\geq \|w_n\|^2_{H^1(\R^3)},$$\\ and so it follows that $(w_n)_{n\in\N}$ is bounded. If $p\in(2,3)$, setting $\alpha_n=\int_{\R^3} (|\nabla w_n|^2+ w_n^2)$, $\gamma_n=\int_{\R^3} \rho\phi_{w_n} w_n^2 $, and $\delta_n= \int_{\R^3} w_n^{p+1}$, and arguing as in Theorem \ref{theorem existence coercive} Claim 2, we see that $\alpha_n$, $\gamma_n$, and $\delta_n$ satisfy the system \eqref{mainSOE} with $d_n:=I(w_n)$ in the place of $c_{\mu_n}$. Thus, solving this system and arguing as in Theorem \ref{theorem existence coercive} Claim 2, we can obtain that $\alpha_n$, $\gamma_n$, and $\delta_n$ are all bounded since $(d_n)_{n\in\N}$ is uniformly bounded. It follows that $(w_n)_{n\in\N}$ is also bounded in this case. Therefore, for all $p\in(2,3]$, there exists $w_0\in E$ such that, up to a subsequence, $w_n\rightharpoonup w_0$ in $E$. Arguing as in the proof of Theorem \ref{theorem existence coercive} Claim 1, we can show $w_n\to w_0$ in $E$, $I(w_0)=c^*$, and $I'(w_0)=0$. We note that by Lemma \ref{lower bound energy nontriv sols}, it holds that $c^* \geq C >0$ for some uniform constant $C>0$, and so $w_0$ is nontrivial. Finally, reasoning as in the conclusion of Theorem \ref{theorem existence coercive}, we see that both $w_0,\phi_{w_0}$ are positive, and this concludes the proof.
\end{proof}

\medskip

\section{Existence: the case of non-coercive $\rho(x)$} \label{non-coercive section} We now turn our attention to the problem of finding solutions when $\rho$ is non-coercive, namely when $\rho(x) \to \rho_{\infty}>0$ as $|x|\to +\infty$. In this setting, $E(\R^3)$ coincides with the larger space $H^1(\R^3)$, and so we look for solutions $(u,\phi_u)\in H^1(\R^3) \times D^{1,2}(\R^3)$ of \eqref{main SP system}. 

\subsection{Bounded PS sequences: proof of Proposition \ref{splitting theorem}}

\noindent Before moving forward, we will need some useful preliminary lemmas.\\

\begin{lemma}[\cite{Mercuri and Willem}]\label{MW 1}
Let $p\geq 0$ and $(u_n)_{n\in\N} \subset L^{p+1}(\R^3)$ be such that $u_n\to u$ almost everywhere on $\R^3$, $\sup_n||u_n||_{L^{p+1}}<+\infty$, and $(u_n)_-\to 0$ in $L^{p+1}(\R^3)$. Then, $u\in L^{p+1}(\R^3)$, $u\geq0$,

\[(u_n-u)_- \to 0  \qquad \textrm{in}\,\,     L^{p+1}(\R^3),\]
and 
\[||(u_n-u)_+||_{L^{p+1}}^{p+1}=||(u_n)_+||_{L^{p+1}}^{p+1}-||u_+||_{L^{p+1}}^{p+1}+o(1).\]\\
\end{lemma}

\begin{lemma}\label{MW 2} 
Let $p>0$ and set 

\[F(u)=\frac{1}{p+1}\int_{\R^3}|u|^{p+1},\quad F_+(u)=\frac{1}{p+1}\int_{\R^3}u_+^{p+1}.\]\\
Assume $(u_n)_{n\in\N}\subset H^1(\R^3)$ is such that $u_n\to u$ almost everywhere on $\R^3$ and $\sup_n||u_n||_{H^1}<+\infty$. Then, it holds that

\[F'(u_n)-F'(u_n-u)-F'(u)=o(1),\qquad \textrm{in}\,\, H^{-1}(\R^3).\]\\
If, in addition, $(u_n)_-\to 0$ in $L^{p+1}(\R^3)$, then

\[F_+'(u_n)-F_+'(u_n-u)-F_+'(u)=o(1), \qquad \textrm{in}\,\, H^{-1}(\R^3). \]
\end{lemma}
\begin{proof}
The result follows as a consequence of Lemma 3.2 in \cite{Mercuri and Willem}, Lemma \ref{MW 1}, and H\"{o}lder's inequality.
\end{proof}

\medskip

The final preliminary result that we need is a splitting lemma for the nonlocal part of the derivative of the energy functional along bounded sequences. The proof follows by convexity estimates and Fatou's lemma, adapting similar arguments of Section $3$ in \cite{Mercuri and Squassina} and Lemma 4.2 in \cite{Degiovanni and Lancelotti} to a nonlocal context. \\

\begin{lemma}\label{dual convergence for splitting}[{\bf Nonlocal splitting lemma}] Assume $(u_n)_{n\in\N}\subset H^1(\R^3)$ is bounded and $u_n \to v_0$ almost everywhere. Suppose further $\rho \in C(\R^3)$ is nonnegative and $\rho(x) \to \rho_{\infty}\geq 0$ as $|x| \to +\infty$. Then, the following hold:\\
\begin{enumerate}[(i)]
\item $\rho\phi_{(u_n-v_0)}(u_n-v_0) - \rho_{\infty}\bar\phi_{(u_n-v_0)}(u_n-v_0) =o(1)$ in $H^{-1}(\R^3)$\\
\item $\rho \phi_{u_n} u_n -\rho \phi_{(u_n-v_0)} (u_n-v_0) - \rho \phi_{v_0} v_0 =o(1)$ in $H^{-1}(\R^3)$.\\
\end{enumerate}
\end{lemma}

\begin{proof}
For the proof of $(i)$, we set 

\[\phi_u^*(x)\coloneqq \int_{\R^3} \frac{u^2(u)}{4\pi |x-y|} \dif y.\]\\
Take any $h\in H^1$, and note that 

\begin{align}\label{dual proof (i) 1}
\bigg|\int_{\R^3} \big(\rho\phi_{(u_n-v_0)}(u_n-v_0) - \rho_{\infty}\bar\phi_{(u_n-v_0)}(u_n-v_0)\big)h\bigg| &\leq \left|\int_{\R^3} (\rho-\rho_{\infty})\phi_{(u_n-v_0)}(u_n-v_0)h\right|\nonumber \\
& \quad+ \left|\int_{\R^3}  \rho_{\infty}(\phi_{(u_n-v_0)}-\bar\phi_{(u_n-v_0)})(u_n-v_0)h\right|  \nonumber\\
& =: I_1+ I_2.\\
\nonumber
\end{align}
Now, by assumption, for every $\epsilon>0$, there exists $R_{\epsilon}>0$ such that $|\rho-\rho_{\infty}|< \epsilon$ for all $|x|>R_{\epsilon}$. So, using H\"{o}lder's and Sobolev's inequalities, we can see that

\begin{align}\label{dual proof (i) 2}
I_1 &\leq \left|\int_{B_{R_{\epsilon}}} (\rho-\rho_{\infty})\phi_{(u_n-v_0)}(u_n-v_0)h\right|+\left|\int_{|x|>R_{\epsilon}} (\rho-\rho_{\infty})\phi_{(u_n-v_0)}(u_n-v_0)h\right| \nonumber\\
&\leq ||\rho||_{L^{\infty}}||\phi_{(u_n-v_0)}||_{L^6}||u_n-v_0||_{L^2(B_{R_{\epsilon}})}||h||_{L^3} + \epsilon ||\phi_{(u_n-v_0)}||_{L^6}||u_n-v_0||_{L^2}||h||_{L^3} \nonumber\\
&\lesssim (||\rho||_{L^{\infty}}||\nabla \phi_{(u_n-v_0)}||_{L^2}||u_n-v_0||_{L^2(B_{R_{\epsilon}})}+ \epsilon||\nabla \phi_{(u_n-v_0)}||_{L^2}||u_n-v_0||_{L^2})||h||_{H^1}. \\
\nonumber
\end{align}
Moreover, by using H\"{o}lder's and Sobolev's inequalities once again, we have 

\begin{align}\label{dual proof (i) 3}
I_2 &\leq \rho_{\infty}||\phi_{(u_n-v_0)}-\bar\phi_{(u_n-v_0)}||_{L^6}||u_n-v_0||_{L^2}||h||_{L^3}\nonumber \\
&\lesssim \rho_{\infty}||\phi_{(u_n-v_0)}-\bar\phi_{(u_n-v_0)}||_{L^6}||u_n-v_0||_{L^2}||h||_{H^1},\\
\nonumber
\end{align} 
and, by Minkowski's, Sobolev's, and Hardy-Littlewood-Sobolev inequalities, for every $\epsilon>0$, it holds

\begin{align}\label{dual proof (i) 4}
||\phi_{(u_n-v_0)}-\bar\phi_{(u_n-v_0)}||_{L^6} &=\left(\int_{\R^3} \left| \int_{\R^3} \frac{(\rho(y)-\rho_{\infty})(u_n-v_0)^2(y)}{4\pi|x-y|}\dif y\right|^6\dif x\right)^{\frac{1}{6}}\nonumber \\
&\leq \left(\int_{\R^3} \left( \int_{B_{R_{\epsilon}}} \frac{|\rho(y)-\rho_{\infty}|(u_n-v_0)^2(y)}{4\pi|x-y|}\dif y\right)^6\dif x\right)^{\frac{1}{6}}\nonumber\\
&\qquad\qquad\qquad\qquad+ \left(\int_{\R^3} \left( \int_{|x|>{R_{\epsilon}}} \frac{|\rho(y)-\rho_{\infty}|(u_n-v_0)^2(y)}{4\pi|x-y|}\dif y\right)^6\dif x\right)^{\frac{1}{6}} \nonumber\\
&\leq ||\rho||_{L^{\infty}}\left(\int_{\R^3} \left( \int_{\R^3} \frac{(u_n-v_0)^2(y)\chi_{B_{R_{\epsilon}}}^2(y)}{4\pi|x-y|}\dif y\right)^6\dif x\right)^{\frac{1}{6}}\nonumber\\
&\qquad\qquad\qquad\qquad+ \epsilon \left(\int_{\R^3} \left( \int_{\R^3} \frac{(u_n-v_0)^2(y)}{4\pi|x-y|}\dif y\right)^6\dif x\right)^{\frac{1}{6}}\nonumber\\
&= ||\rho||_{L^{\infty}}||\phi_{(u_n-v_0)\chi_{B_{R_{\epsilon}}}}^*||_{L^6} + \epsilon ||\phi_{(u_n-v_0)}^*||_{L^6}\nonumber\\
&\lesssim ||\rho||_{L^{\infty}}||\nabla\phi_{(u_n-v_0)\chi_{B_{R_{\epsilon}}}}^*||_{L^2} + \epsilon ||\nabla \phi_{(u_n-v_0)}^*||_{L^2}\nonumber\\
&\lesssim ||\rho||_{L^{\infty}}||(u_n-v_0)\chi_{B_{R_{\epsilon}}}||_{L^{\frac{12}{5}}}^2 + \epsilon ||\nabla \phi_{(u_n-v_0)}^*||_{L^2}.\\
\nonumber
\end{align}
So, putting together \eqref{dual proof (i) 1}, \eqref{dual proof (i) 2}, \eqref{dual proof (i) 3}, and \eqref{dual proof (i) 4}, we obtain, for every $\epsilon>0$,

\begin{align}
\bigg|\int_{\R^3}& \big(\rho\phi_{(u_n-v_0)}(u_n-v_0) - \rho_{\infty}\bar\phi_{(u_n-v_0)}(u_n-v_0)\big)h\bigg| \nonumber \\ 
&\leq C (||\rho||_{L^{\infty}}||\nabla \phi_{(u_n-v_0)}||_{L^2}||u_n-v_0||_{L^2(B_{R_{\epsilon}})}+ \epsilon||\nabla \phi_{(u_n-v_0)}||_{L^2}||u_n-v_0||_{L^2} \nonumber \\
&\quad\quad+ \rho_{\infty}||\rho||_{L^{\infty}}||(u_n-v_0)\chi_{B_{R_{\epsilon}}}||_{L^{\frac{12}{5}}}^2||u_n-v_0||_{L^2} +  \rho_{\infty}\epsilon||\nabla \phi_{(u_n-v_0)}^*||_{L^2}||u_n-v_0||_{L^2})||h||_{H^1},\nonumber\\
\nonumber
\end{align}
for some $C>0$. Since $\rho\in L^{\infty}$, $\phi_{(u_n-v_0)}, \phi_{(u_n-v_0)}^*$ are uniformly bounded in $D^{1,2}$, $u_n-v_0$ is uniformly bounded in $L^2$, and $u_n-v_0 \to 0$ in $L^{2}_{\textrm{loc}}$ and $L^{\frac{12}{5}}_{\textrm{loc}}$, then we have proven $(i)$.\\

To prove $(ii)$, we first take any $h\in H^1$, and note that by H\"{o}lder's and Sobolev's inequalities, it holds that\

\begin{align}\label{second part dual lemma 1}
\bigg|\int_{\R^3} (\rho \phi_{u_n} u_n &-\rho \phi_{(u_n-v_0)} (u_n-v_0) - \rho \phi_{v_0} v_0 )h\bigg| \nonumber\\
&\leq ||\rho||_{L^{\infty}} ||\phi_{u_n} u_n - \phi_{(u_n-v_0)} (u_n-v_0) -  \phi_{v_0} v_0 ||_{L^{\frac{3}{2}}}||h||_{L^3}\nonumber\\
&\leq C||\rho||_{L^{\infty}} || \phi_{u_n} u_n - \phi_{(u_n-v_0)} (u_n-v_0) -  \phi_{v_0} v_0 ||_{L^{\frac{3}{2}}}||h||_{H^1},\\
\nonumber
\end{align}
for some $C>0$. Now, by convexity, iterating the inequality

\[ |a+b|^{\frac{3}{2}} \leq \sqrt{2}\left( |a|^{\frac{3}{2}}+ |b|^{\frac{3}{2}}\right),\]\\
we can obtain

\begin{align}\label{second part dual lemma 2}
F_n &\coloneqq \left|  \phi_{u_n} u_n - \phi_{(u_n-v_0)} (u_n-v_0) -  \phi_{v_0} v_0 \right|^{\frac{3}{2}} \nonumber\\
&\leq 2\left(\left|  \left(\phi_{u_n}-\phi_{(u_n-v_0)}\right) u_n \right|^{\frac{3}{2}}+\left| \phi_{(u_n-v_0)} v_0 \right|^{\frac{3}{2}} +\left| \phi_{v_0} v_0 \right|^{\frac{3}{2}} \right). \\
\nonumber
\end{align}
Then, using the Cauchy-Schwarz inequality, we notice that 

\begin{align}
\left|  (\phi_{u_n}-\phi_{(u_n-v_0)}\right| 
&\leq \int_{\R^3} \frac{ \rho |2u_n-v_0||v_0|}{4\pi|x-y|}\dif y\nonumber \\
&\leq \left( \int_{\R^3} \frac{ \rho |2u_n-v_0|^2}{4\pi|x-y|}\dif y \right)^{\frac{1}{2}} \left(\int_{\R^3} \frac{ \rho |v_0|^2}{4\pi|x-y|}\dif y\right)^{\frac{1}{2}}\nonumber\\
&=\phi_{(2u_n-v_0)}^{\frac{1}{2}}\phi_{v_0}^{\frac{1}{2}},\nonumber \\
\nonumber
\end{align}
and so, using this and applying Young's inequality twice, we see that, for every $\epsilon>0$,

\begin{align}\label{second part dual lemma 3}
\left|  \left(\phi_{u_n}-\phi_{(u_n-v_0)}\right) u_n \right|^{\frac{3}{2}}&\leq \phi_{(2u_n-v_0)}^{\frac{3}{4}}\phi_{v_0}^{\frac{3}{4}}|u_n|^{\frac{3}{2}} \nonumber\\
&\leq \epsilon^{\frac{8}{7}}\phi_{(2u_n-v_0)}^{\frac{6}{7}}|u_n|^{\frac{12}{7}} + {\epsilon^{-8}}\phi_{v_0}^6\nonumber\\
&\leq \epsilon^{\frac{8}{7}} \left( \phi_{(2u_n-v_0)}^{6} +|u_n|^2\right)+ {\epsilon^{-8}}\phi_{v_0}^6.\\
\nonumber
\end{align}
Moreover, again using Young's inequality, it holds, for every $\epsilon>0$,

\begin{align}\label{second part dual lemma 4}
\left| \phi_{(u_n-v_0)} v_0 \right|^{\frac{3}{2}}\leq \epsilon^4 \phi_{(u_n-v_0)}^6 +{\epsilon^{-\frac{4}{3}}}v_0^2.\\
\nonumber
\end{align}
Combining \eqref{second part dual lemma 2}, \eqref{second part dual lemma 3}, and \eqref{second part dual lemma 4}, we see that, for all $\epsilon>0$, 

\begin{align}
F_n &\leq 2\left(\epsilon^{\frac{8}{7}} \left( \phi_{(2u_n-v_0)}^{6} +|u_n|^2\right)+ {\epsilon^{-8}}\phi_{v_0}^6+\epsilon^4 \phi_{(u_n-v_0)}^6 +{\epsilon^{-\frac{4}{3}}}v_0^2+\left| \phi_{v_0} v_0 \right|^{\frac{3}{2}} \right) =: G_n, \nonumber\\
\nonumber
\end{align}
and so $G_n-F_n \geq0$. We recall that by assumption $u_n\to v_0$ almost everywhere, and so it follows that $\phi_{(u_n-v_0)} \to 0,$  $\phi_{u_n} \to \phi_{v_0}, $ and $\phi_{(2u_n-v_0)} \to \phi_{v_0}$ almost everywhere. Thus, applying Fatou's Lemma to $G_n-F_n$, we obtain

\begin{align}
2\int_{\R^3} &\left(\epsilon^{\frac{8}{7}} (\phi_{v_0}^{6} +|v_0|^2)+ {\epsilon^{-8}}\phi_{v_0}^6 +{\epsilon^{-\frac{4}{3}}}v_0^2+\left| \phi_{v_0} v_0 \right|^{\frac{3}{2}}\right) \nonumber \\
&\leq  2\bigg(\epsilon^{\frac{8}{7}}\sup_{n\geq1} \int_{\R^3} \left(\phi_{(2u_n-v_0)}^{6} +|u_n|^2\right) + {\epsilon^{-8}}\int_{\R^3}\phi_{v_0}^6+\epsilon^4  \sup_{n\geq1} \int_{\R^3}\phi_{(u_n-v_0)}^6 +{\epsilon^{-\frac{4}{3}}}\int_{\R^3} v_0^2\nonumber\\
& \qquad\qquad+\int_{\R^3}\left| \phi_{v_0} v_0 \right|^{\frac{3}{2}}\bigg)-\limsup_{n\to+\infty} \int_{\R^3} F_n. \nonumber\\
\nonumber
\end{align}
Therefore, after cancelations and using Sobolev's inequality, we see that

\begin{align}
\limsup_{n\to+\infty} \int_{\R^3} F_n &\leq 2\bigg(\epsilon^{\frac{8}{7}}\sup_{n\geq1} \int_{\R^3} \left(\phi_{(2u_n-v_0)}^{6} +|u_n|^2\right) +\epsilon^4  \sup_{n\geq1}\int_{\R^3}\phi_{(u_n-v_0)}^6 -\epsilon^{\frac{8}{7}} \int_{\R^3}\left(\phi_{v_0}^{6}+|v_0|^2\right) \bigg)\nonumber\\
&=2 \bigg(\epsilon^{\frac{8}{7}}\sup_{n\geq1} \left(||\phi_{(2u_n-v_0)}||_{L^6}^{6} +||u_n||_{L^2}^2\right) +\epsilon^4  \sup_{n\geq1}||\phi_{(u_n-v_0)}||_{L^6}^6 \nonumber\\
&\qquad- \epsilon^{\frac{8}{7}} \left(||\phi_{v_0}||_{L^6}^6+||v_0||_{L^2}^2\right)\bigg)\nonumber\\
&\leq C \bigg(\epsilon^{\frac{8}{7}}\sup_{n\geq1} \left(||\nabla\phi_{(2u_n-v_0)}||_{L^2}^{6} +||u_n||_{L^2}^2\right) +\epsilon^4  \sup_{n\geq1}||\nabla\phi_{(u_n-v_0)}||_{L^2}^6 \nonumber\\
&\qquad- \epsilon^{\frac{8}{7}} \left(||\phi_{v_0}||_{L^6}^6+||v_0||_{L^2}^2\right)\bigg),\nonumber\\
\nonumber
\end{align}
for some $C>0$ and for all $\epsilon>0$. We note that $u_n, v_0$ are uniformly bounded in $L^2$ and $\phi_{(u_n-v_0)}, \phi_{(2u_n-v_0)}$ are uniformly bounded in $D^{1,2}$ since $u_n-v_0, 2u_n-v_0$ are uniformly bounded in $H^1$. Moreover, since $v_0\in H^1$, it follows that $||\phi_{v_0}||_{L^6}^6$ is bounded by Sobolev's inequality. Hence, since $\epsilon>0$ is arbitrary, it holds that

\[\lim_{n\to+\infty} \int_{\R^3} F_n =0,\]\\
which combined with \eqref{second part dual lemma 1} yields $(ii),$ and this concludes the proof.
\end{proof}

\medskip 

With these preliminaries in place, we now prove a useful `splitting' proposition for bounded Palais-Smale sequences for $I_{\mu}$, that highlights the connection to the problem at infinity. \\

\begin{proof} [Proof of Proposition \ref{splitting theorem}]
Since $(u_n)_{n\in\N}$ is bounded in $H^1$, we may assume $u_n\rightharpoonup v_0$ in $H^1$ and $u_n\to v_0$ a.e. in $\R^3$. We set $u_n^1 \coloneqq u_n-v_0$, and we first note that

\begin{equation}\label{first norm splitting}
||u_n^1||_{H^1}^2=||u_n-v_0||_{H^1}^2=||u_n||_{H^1}^2-||v_0||_{H^1}^2 +o(1).
\end{equation}\\
We now prove three claims involving the sequence $(u_n^1)_{n\in\N}$.\\

\noindent \textbf{Claim 1.}
\emph{$I_{\mu}^{\infty}(u_n^1)= I_{\mu}(u_n)-I_{\mu}(v_0) +o(1)$.}\\

\noindent Testing $I_{\mu}'(u_n)$ with $(u_n)_-$ we have 

\begin{align*}
I_{\mu}'(u_n)(u_n)_- &= \int_{\R^3}\left(\nabla u_n\nabla ((u_n)_-)+u_n(u_n)_-\right) +\int_{\R^3} \rho \phi_{u_n}u_n(u_n)_- -{\mu}\int_{\R^3}(u_n)_+^p(u_n)_-\\
&=||(u_n)_-||^2_{H^1}+ \int_{\R^3} \rho \phi_{u_n}(u_n)_-^2. \\
\end{align*}
Since $(u_n)_{n\in\N}$ is bounded, $I_{\mu}'(u_n)(u_n)_- =o(1),$ which implies 

\[(u_n)_-\to 0 \text{ in } H^1,\]\\
and by Sobolev's embedding 

\[(u_n)_-\to 0 \text{ in } L^{p+1} \,\, \forall p\in[1,5].\]\\
Now, using this and the boundedness of $(u_n)_{n\in\N}$ in $L^{p+1}$, it holds, by Lemma \ref{MW 1}, that 

\[||(u_n^1)_+||_{L^{p+1}}^{p+1}=||(u_n)_+||_{L^{p+1}}^{p+1}-||(v_0)_+||_{L^{p+1}}^{p+1} +o(1).\]\\
Therefore, using this and \eqref{first norm splitting}, we can see that 

\begin{align}\label{first split of I mu infinity}
I_{\mu}^{\infty}(u_n^1)&= \frac{1}{2}(||u_n||_{H^1}^2-||v_0||_{H^1}^2) +\frac{1}{4}\int_{\R^3}\rho_{\infty}\bar\phi_{(u_n-v_0)}(u_n-v_0)^2 \\
&\qquad -\frac{\mu}{p+1}\left(||(u_n)_+||_{L^{p+1}}^{p+1}-||(v_0)_+||_{L^{p+1}}^{p+1}\right) +o(1).\nonumber\\
\nonumber
\end{align}
\noindent We now notice that since, by symmetry, 

\begin{equation*}
\int_{\R^3} \rho_{\infty}(u_n-v_0)^2\phi_{(u_n-v_0)}= \int_{\R^3} \rho(u_n-v_0)^2\bar\phi_{(u_n-v_0)},\\
\end{equation*} \\
then it holds that

\begin{align*}
\left| \int_{\R^3}\rho \phi_{(u_n-v_0)}(u_n-v_0)^2 - \int_{\R^3}\rho_{\infty}\bar\phi_{(u_n-v_0)}(u_n-v_0)^2 \right| &\leq   \int_{\R^3} \phi_{(u_n-v_0)}(u_n-v_0)^2 |\rho(x)-\rho_{\infty}|  \\
& \qquad + \int_{\R^3}\bar\phi_{(u_n-v_0)}(u_n-v_0)^2 |\rho(x)-\rho_{\infty}| \\
\\
&=: I_1 +I_2.\\
\end{align*}
We note that for all $\epsilon>0$ there exists $R_{\epsilon}>0$ such that $|\rho(x)-\rho_{\infty}|<\epsilon$ for all $|x|>R_{\epsilon}$. Thus, we can see,

\begin{align*}
I_1&\leq  \int_{B_{R_{\epsilon}}} \phi_{(u_n-v_0)}(u_n-v_0)^2 |\rho(x)-\rho_{\infty}| + \int_{|x|>R_{\epsilon}} \phi_{(u_n-v_0)}(u_n-v_0)^2 |\rho(x)-\rho_{\infty}| \\
&\leq C\left( ||\rho||_{L^{\infty}}|| \,||\nabla \phi_{(u_n-v_0)}||_{L^2}||u_n-v_0||^2_{L^{\frac{12}{5}}(B_{R_\epsilon})} +\epsilon \,||\nabla \phi_{(u_n-v_0)}||_{L^2}||u_n-v_0||^2_{L^{\frac{12}{5}}}\right),\\
\end{align*}
where $C>0$ is a constant. Since $\rho\in L^{\infty}$, $\phi_{(u_n-v_0)}$ is uniformly bounded in $D^{1,2}$ and $u_n-v_0 \to 0$ in $L^{12/5}_{\textrm{loc}}$, the above shows that $I_1 \to 0$ as $n\to +\infty$. Similarly, we can see that

\begin{align*}
I_2\leq  C' \left( ||\rho||_{L^{\infty}}|| \,||\nabla \bar\phi_{(u_n-v_0)}||_{L^2}||u_n-v_0||^2_{L^{\frac{12}{5}}(B_{R_\epsilon})} +\epsilon ||\nabla \bar\phi_{(u_n-v_0)}||_{L^2}||u_n-v_0||^2_{L^{\frac{12}{5}}}\right),\\
\end{align*}
and so $I_2\to 0$ as $n\to +\infty$. Therefore, we have shown that 

\[\int_{\R^3}\rho_{\infty}\bar\phi_{(u_n-v_0)}(u_n-v_0)^2 = \int_{\R^3}\rho \phi_{(u_n-v_0)}(u_n-v_0)^2 +o(1),\]\\
and thus, by the nonlocal Brezis-Lieb Lemma \ref{nonlocalBL}, it holds that 

\[\int_{\R^3}\rho_{\infty}\bar\phi_{(u_n-v_0)}(u_n-v_0)^2 =\int_{\R^3}\rho \phi_{u_n}u_n^2 -\int_{\R^3}\rho \phi_{v_0}v_0^2 +o(1).\]\\
Putting this together with \eqref{first split of I mu infinity}, we see that $I_{\mu}^{\infty}(u_n^1)= I_{\mu}(u_n)-I_{\mu}(v_0) +o(1)$, and the claim is proved.\\ 

\noindent \textbf{Claim 2.}
\emph{$I_{\mu}'(v_0)=0$ and $v_0\geq 0$.}\\

 We notice that for all $\psi\in C^{\infty}_c(\R^3)$, it holds that
 
\[I_{\mu}'(u_n)(\psi)=\int_{\R^3} (\nabla u_n \nabla \psi + u_n \psi) +\int_{\R^3} \rho \phi_{u_n} u_n \psi -\mu\int_{\R^3} (u_n)_+^p\psi.\]\\
Using the fact that $u_n\rightharpoonup v_0$ in $H^1$ and a local compactness argument, we have $I_{\mu}'(u_n)(\psi) = I_{\mu}'(v_0)(\psi) +o(1)$. So, since $I_{\mu}'(u_n) \to 0$ by the definition of a Palais-Smale sequence, it holds that $I_{\mu}'(v_0)=0$ by density. We note that by testing this equation with $(v_0)_-$, we obtain that $v_0\geq0$.\\

\noindent \textbf{Claim 3.}
\emph{$(I_{\mu}^{\infty})'(u_n^1)\to 0$.} \\

\noindent We first note that by Lemma \ref{dual convergence for splitting}, it holds that 

\begin{align}\label{poisson convergence in dual}
\rho \phi_{u_n} u_n & -\rho_{\infty} \bar\phi_{(u_n-v_0)} (u_n-v_0) - \rho \phi_{v_0} v_0 \nonumber\\
&=\rho \phi_{u_n} u_n -\rho \phi_{(u_n-v_0)} (u_n-v_0) - \rho \phi_{v_0} v_0+o(1)\\
&=o(1) \qquad \textrm{in}\,\, H^{-1}(\R^3).\nonumber\\
\nonumber
\end{align}
\noindent Moreover, since we have showed in Claim 1 that $(u_n)_-\to 0$ in $L^{p+1}$, then, by Lemma \ref{MW 2}, it follows that

\begin{equation}\label{up in dual}
(u_n)_+^p -(u_n-v_0)_+^p - (v_0)^p =o(1), \qquad \textrm{in}\,\, H^{-1}(\R^3).\\
\end{equation}\\
\noindent Therefore, using \eqref{poisson convergence in dual} and \eqref{up in dual}, we can conclude that 

\[(I_{\mu}^{\infty})'(u_n^1) = I_{\mu}' (u_n)- I_{\mu}' (v_0) + o(1),\]\\
and so
\[(I_{\mu}^{\infty})'(u_n^1)=o(1)\]\\
since $I_{\mu}' (u_n)\to 0$ by the definition of Palais-Smale sequence and $I_{\mu}' (v_0)=0$ by Claim 2. This completes the proof of the claim.\\

\noindent \textbf{Partial conclusions.} With these results in place, we now define 

\[\delta \coloneqq \limsup_{n\to+\infty}\left(\sup_{y\in\R^3}\int_{B_1(y)}|u_n^1|^{p+1}\right).\]\\
\noindent We can see that $\delta\geq0$. If $\delta=0$, the P.\ L.\ Lions Lemma \cite{Lions} implies $u_n^1\to 0$ in $L^{p+1}$. Since it holds that

\[(I_{\mu}^{\infty})'(u_n^1)(u_n^1)=||u_n^1||_{H^1}^2+\int_{\R^3} \rho_{\infty}\bar\phi_{u_n^1}(u_n^1)^2-\mu\int_{\R^3}(u_n^1)_+^{p+1},\]\\
and $(I_{\mu}^{\infty})'(u_n^1)\to 0$ by Claim 3, then, if $u_n^1\to 0$ in $L^{p+1}$, it follows that $u_n^1\to 0$ in $H^1$. In this case, we are done since we have $u_n\to v_0$ in $H^1$. Therefore, we assume $\delta>0$. This implies that there exists $(y_n^1)_{n\in\N}\subset\R^3$ such that 

\[\int_{B_1(y_n^1)}|u_n^1|^{p+1}>\frac{\delta}{2}.\]\\
We now define $v_n^1\coloneqq u_n^1(\cdot+y_n^1)$. We my assume $v_n^1\rightharpoonup v_1$ in $H^1$ and $v_n^1\to v_1$ a.e. in $\R^3$. Then, since 

\[\int_{B_1(0)}|v_n^1|^{p+1}>\frac{\delta}{2},\]\\
it follows from Rellich Theorem that $v_1\not\equiv 0$. Since $u_n^1\rightharpoonup0$ in $H^1$, then $(y_n^1)_{n\in\N}$ must be unbounded and so we assume, up to a subsequence, $|y_n^1|\to+\infty$. We set $u_n^2\coloneqq u_n^1-v_1(\cdot-y_n^1)$, and, using \eqref{first norm splitting}, we note that 

\begin{equation}\label{second norm splitting}
||u_n^2||_{H^1}^2=||u_n^1||_{H^1}^2-||v_1||_{H^1}^2 +o(1)=||u_n||_{H^1}^2-||v_0||_{H^1}^2-||v_1||_{H^1}^2 +o(1).
\end{equation}\\
We now prove three claims regarding the sequence $(u_n^2)_{n\in\N}$.\\

\noindent \textbf{Claim 4.}
\emph{$I_{\mu}^{\infty}(u_n^2)=  I_{\mu}(u_n)-I_{\mu}(v_0)-I_{\mu}^{\infty}(v_1) +o(1)$.}\\

Arguing similarly as in Claim 1, namely testing $I_{\mu}^{\infty}(u_n^1)$ with $(u_n^1)_-$, we can show that $(u_n^1)_- \to 0$ in $L^{p+1}$, and so $(u_n^1(\cdot+y_n^1))_- \to 0$ in $L^{p+1}$. Thus, once again using Lemma \ref{MW 1}, we can see that  

\begin{align*}
||(u_n^1)_+||_{L^{p+1}}^{p+1}
&=||(u_n^1(\cdot+y_n^1)-v_1)_+||_{L^{p+1}}^{p+1} +||(v_1)_+||_{L^{p+1}}^{p+1}+o(1)\\
&=||(u_n^1-v_1(\cdot-y_n^1))_+||_{L^{p+1}}^{p+1} +||(v_1)_+||_{L^{p+1}}^{p+1}+o(1)\\
&=||(u_n^2)_+||_{L^{p+1}}^{p+1} +||(v_1)_+||_{L^{p+1}}^{p+1}+o(1),\\
\end{align*}
and so

\[||(u_n^2)_+||_{L^{p+1}}^{p+1} =||(u_n^1)_+||_{L^{p+1}}^{p+1}-||(v_1)_+||_{L^{p+1}}^{p+1}+o(1).\]\\
\noindent Therefore, using this and \eqref{second norm splitting}, we have that 

\begin{align*}
I_{\mu}^{\infty}(u_n^2)&=||u_n^1||_{H^1}^2-||v_1||_{H^1}^2 +\frac{1}{4}\int_{\R^3} \rho_{\infty}\bar\phi_{(u_n^1-v_1(\cdot-y_n^1))}(u_n^1-v_1(x-y_n^1))^2\\
&\qquad-\frac{\mu}{p+1}\left(||(u_n^1)_+||_{L^{p+1}}^{p+1}-||(v_1)_+||_{L^{p+1}}^{p+1}\right)+o(1).\\
\end{align*}
\noindent We can show, by changing variables and using Lemma \ref{nonlocalBL}, that

\begin{align*}
\int_{\R^3} \rho_{\infty}\bar\phi_{(u_n^1-v_1(\cdot-y_n^1))}(u_n^1-v_1(x-y_n^1))^2
&=\int_{\R^3} \rho_{\infty}\bar\phi_{u_n^1}(u_n^1)^2-\int_{\R^3} \rho_{\infty}\bar\phi_{v_1}v_1^2+o(1).\\
\end{align*}
\noindent Thus, by combining the last two equations and using Claim 1, we see that $I_{\mu}^{\infty}(u_n^2)= I_{\mu}^{\infty}(u_n^1)-I_{\mu}^{\infty}(v_1) +o(1)=I_{\mu}(u_n)-I_{\mu}(v_0)-I_{\mu}^{\infty}(v_1) +o(1)$, and so the claim is proved.\\

\noindent \textbf{Claim 5.}
\emph{$(I_{\mu}^{\infty})'(v_1)=0$ and $v_1\geq0$.}\\

Let $h\in H^1(\R^3)$ and set $h_n\coloneqq h(\cdot-y_n^1)$. By a change of variables, we can see that

\begin{align*}
(I_{\mu}^{\infty})'(u_n^1(x+y_n^1))(h)
&=(I_{\mu}^{\infty})'(u_n^1)(h_n),\\
\end{align*}
\noindent and so, since $(I_{\mu}^{\infty})'(u_n^1)\to 0$ by Claim 3, we have that 

\begin{equation}\label{convergence deriv functional}
(I_{\mu}^{\infty})'(u_n^1(x+y_n^1))\to 0.
\end{equation}\\
\noindent We now note, for any $\psi\in C^{\infty}_c(\R^3)$, it holds that

\begin{align*}
\bigg|\int_{\R^3} &\rho_{\infty}\bar\phi_{u_n^1}(x+y_n^1)u_n^1(x+y_n^1)\psi-\int_{\R^3} \rho_{\infty}\bar\phi_{v_1}v_1\psi \bigg|\\
&\leq\left|\int_{\R^3} \rho_{\infty}\bar\phi_{u_n^1}(x+y_n^1)(u_n^1(x+y_n^1)-v_1)\psi\right| +\left|\int_{\R^3} \rho_{\infty}(\bar\phi_{u_n^1}(x+y_n^1)-\bar\phi_{v_1})v_1\psi\right|\\
&\leq \rho_{\infty}||\bar\phi_{u_n^1}(\cdot+y_n^1)||_{L^6}||u_n^1(\cdot+y_n^1)-v_1||_{L^2(\textrm{supp}\, \psi)}||\psi||_{L^3}  \\
&\qquad\qquad+\rho_{\infty}||\bar\phi_{u_n^1}(\cdot+y_n^1)-\bar\phi_{v_1}||_{L^2(\textrm{supp}\, \psi)}||v_1||_{L^6}||\psi||_{L^3}, \\
\end{align*}
\noindent and so since $u_n^1(\cdot+y_n^1)-v_1 \to 0$ in $L^{2}_{\textrm{loc}}$ and $\bar\phi_{u_n^1}(\cdot+y_n^1)-\bar\phi_{v_1}\to 0$ in $L^{2}_{\textrm{loc}}$, and all of the other terms in the final equation are bounded, then we have shown that 

\[\int_{\R^3} \rho_{\infty}\bar\phi_{u_n^1}(x+y_n^1)u_n^1(x+y_n^1)\psi \to \int_{\R^3} \rho_{\infty}\bar\phi_{v_1}v_1\psi. \]\\
\noindent Using this and the fact that $u_n^1(\cdot+y_n^1)\rightharpoonup v_1$ in $H^1$, it follows by standard arguments that $(I_{\mu}^{\infty})'(u_n^1(x+y_n^1))(\psi)=(I_{\mu}^{\infty})'(v_1)(\psi)+o(1)$. This implies that $(I_{\mu}^{\infty})'(v_1)=0$, by \eqref{convergence deriv functional} and density. Testing this equation with $(v_1)_-$, shows that $v_1\geq0$.\\

\noindent \textbf{Claim 6.}
\emph{$(I_{\mu}^{\infty})'(u_n^2)\to 0$.} \\

We take any $h \in H^1(\R^3)$ and set $h_n \coloneqq h (\cdot+y_n^1)$. We note that, by a change of variables, it holds that

\begin{align}\label{claim 6 eq 1}
(I_{\mu}^{\infty})'(u_n^2)(h)
&=(I_{\mu}^{\infty})'(u_n^1(\cdot+y_n^1)-v_1)(h_n).\\
\nonumber
\end{align}
\noindent Now, arguing as we did in the proof of $(ii)$ of Lemma \ref{dual convergence for splitting}, we can show that

\begin{equation}\label{poisson convergence in dual 2} 
\rho_{\infty}\bar\phi_{u_n^1}(\cdot+y_n^1)u_n^1(\cdot+y_n^1)-\rho_{\infty}\bar\phi_{(u_n^1(\cdot+y_n^1)-v_1)}(u_n^1(\cdot+y_n^1)-v_1) - \rho_{\infty}\bar\phi_{v_1}v_1 =o(1),\quad \textrm{in}\,\, H^{-1}(\R^3). 
\end{equation}\\
Moreover, since we showed $(u_n^1(\cdot+y_n^1))_- \to 0$ in $L^{p+1}$ in Claim 4, we can once again use Lemma \ref{MW 2} to conclude that

\[ (u_n^1(\cdot+y_n^1))_+^p-(u_n^1(\cdot+y_n^1)-v_1)_+^p - (v_1)_+^p=o(1), \qquad \textrm{in}\,\, H^{-1}(\R^3). \]\\
\noindent It follows from this and \eqref{poisson convergence in dual 2} that 

\begin{equation}\label{claim 6 eq 2}
(I_{\mu}^{\infty})'(u_n^1(\cdot+y_n^1)-v_1)=(I_{\mu}^{\infty})'(u_n^1(\cdot+y_n^1))-(I_{\mu}^{\infty})'(v_1)+o(1) \qquad \textrm{in}\,\, H^{-1}(\R^3) .
\end{equation}\\
Since, by Claim 5 and a change of variables,  it holds that 

\[(I_{\mu}^{\infty})'(u_n^1(\cdot+y_n^1))(h_n)-(I_{\mu}^{\infty})'(v_1)(h_n) = (I_{\mu}^{\infty})'(u_n^1)(h),\]\\
then combining this, \eqref{claim 6 eq 1} and  \eqref{claim 6 eq 2}, we see that

\[(I_{\mu}^{\infty})'(u_n^2)=(I_{\mu}^{\infty})'(u_n^1)+o(1)\qquad \textrm{in}\,\, H^{-1}(\R^3).\]\\
It therefore follows that $(I_{\mu}^{\infty})'(u_n^2)\to 0$ since $(I_{\mu}^{\infty})'(u_n^1)\to 0$ by Claim 3, and we are done.\\

\noindent \textbf{Conclusions.}
With these results in place we can now see that if $u_n^2\to 0$ in $H^1$, then we are done. Otherwise, $u_n^2\rightharpoonup 0$ in $H^1$, but not strongly, and so we repeat the argument. By iterating the procedure, we obtain sequences of points $(y_n^j)_{n\in\N}\subset\R^3$ such that $|y_n^j|\to+\infty$, $|y_n^j-y_n^{j'}|\to+\infty$ as $n\to+\infty$ if $j\neq j'$ and a sequence of functions $(u_n^j)_{n\in\N}$ with $u_n^1=u_n-v_0$ and $u_n^j=u_n^{j-1}-v_{j-1}(\cdot-y_n^{j-1})$ for $j\geq 2$ such that\\
\begin{align}
&u_n^j(x+y_n^j)\rightharpoonup v_j(x) \text{ in } H^1, \nonumber \\
& ||u_n||_{H^1(\R^3)}^2 = \sum_{j=0}^{l-1} ||v_j||_{H^1(\R^3)}^2+||u_n^l||_{H^1}^2 +o(1),\label{finite iteration}\\
& ||u_n-v_0- \sum_{j=1}^{l} v_j(\cdot -y_n^j)||_{H^1(\R^3)}\to 0 \text{ as } n\to+\infty,\nonumber \\
& I_{\mu}(u_n) = I_{\mu}(v_0)+\sum_{j=1}^{l-1}I_{\mu}^{\infty}(v_j)+I_{\mu}^{\infty}(u_n^l)+o(1), \nonumber\\
& (I_{\mu}^{\infty})'(v_j)=0 \text{ and } v_j\geq0 \text{ for } j\geq1,\nonumber\\
\nonumber
\end{align}

\noindent We notice from the last equation that it holds that $(I_{\mu}^{\infty})'(v_j)(v_j)=0$ for each $j\geq1$. Using this, the Sobolev embedding theorem and the fact that $\mu\leq1$, we have that

\[S_{p+1}||v_j||_{L^{p+1}}^2 \leq ||v_j||_{H^{1}}^2 \leq ||v_j||_{H^{1}}^2 + \int_{\R^3} \rho_{\infty}\bar\phi_{v_j} (v_j)^2 = \mu ||(v_j)_+||_{L^{p+1}}^{p+1}\leq ||v_j||_{L^{p+1}}^{p+1},\]\\
\noindent and therefore, we can conclude that, for each $j\geq1$,

\[||v_j||_{H^1}^2\geq(S_{p+1})^{\frac{p+1}{p-1}}.\]\\
Combining this and the fact $(u_n)_{n\in\N}$ is bounded in $H^1$, we see from \eqref{finite iteration} that the iteration must stop at some finite index $l\in\N$.
\end{proof}

\subsection{Proof of Theorem \ref{general theorem}}
We are finally in position to establish two sufficient conditions that guarantee the existence of a mountain pass solution to \eqref{main SP system} in the case of non-coercive $\rho$, if $p\in(2,3)$.\\

\begin{proof} [Proof of Theorem \ref{general theorem}]
We first note that by Lemma \ref{claim3} with $E=H^1$, the set $\mathcal{M}$, defined in \eqref{set of mu}, is nonempty. \\

\noindent \textbf{Claim 1.} \emph{Under assumptions $(i)$, the values $c_{\mu}$ are critical levels of $I_{\mu}$ for all $\mu \in(1-\epsilon,1]\cap \mathcal M$, with $\epsilon>0$ sufficiently small. Namely, there exists a nonnegative ${u} \in H^1$ such that $I_{\mu}({u}) = c_{\mu}$ and $I_{\mu}'({u})= 0$ for all $\mu \in(1-\epsilon,1]\cap \mathcal M$. Under assumptions $(ii)$, the same statement holds for all $\mu \in \mathcal{M}$.\\}



We recall that for all $\mu \in \mathcal{M}$, by definition, there exists a bounded sequence $(u_n)_{n\in\N}\subset H^1$ such that $I_{\mu}(u_n) \to c_\mu$ and $I'_{\mu}(u_n)\to 0$. We note that by Proposition \ref{splitting theorem} and the definition of $(u_n)_{n\in\N}$, it holds that

\begin{equation}\label{energy level sum}
c_{\mu} =I_{\mu}(v_0)+\sum_{j=1}^{l}I_{\mu}^{\infty}(v_j),
\end{equation}\\
\noindent where $v_0$ is a nonnegative solution of \eqref{SP one equation perturbed} and $v_j$ are nonnegative solutions of \eqref{SP one equation perturbed infinity} for $1\leq j \leq l$. \\

 Assume that $(i)$ holds. For $\epsilon>0$ small enough, it holds that $c_\mu<c_\mu^\infty$ for all $\mu\in (1-\epsilon,1]\cap \mathcal M,$ by continuity. Pick $\mu$ on this set. If $v_j$ is nontrivial for some $1\leq j \leq l$, it would follow that $I_{\mu}^{\infty}(v_j) \geq c_{\mu}^{\infty}>c_{\mu}$ by Lemma \ref{lower bound nontriv sols infinity}. This is in contradiction with \eqref{energy level sum} since $I_{\mu}(v_0) \geq 0$, by Lemma \ref{lower bound energy nontriv sols}, and so, $v_j \equiv 0$ for all $1\leq j \leq l$. Therefore, $u_n \to {v_0}$ in $H^1$ by $(iv)$ of Proposition \ref{splitting theorem}, $I_{\mu}(v_0)=c_{\mu}$ by \eqref{energy level sum}, and $I_{\mu}'(v_0)=0$ since $v_0$ is a nonnegative solution of \eqref{SP one equation perturbed}. Thus, we have shown $c_{\mu}$ is a critical level of $I_{\mu}$ in this case.\\

Now, assume that $(ii)$ holds. We note that this implies that $I_{\mu}(\gamma(t))\leq I_{\mu}^{\infty}(\gamma(t))$ for each fixed $\gamma\in \Gamma^{\infty}$, $\mu \in [\frac{1}{2},1]$ and $t\in[0,1]$. It therefore follows that $I_{\frac{1}{2}}(\gamma(1))\leq I_{\frac{1}{2}}^{\infty}(\gamma(1))<0$ for all $\gamma\in \Gamma^{\infty}$, and so $\Gamma^{\infty} \subseteq \Gamma$. Using this and Lemma \ref{lower bound nontriv sols infinity}, we can see that for each nontrivial $v_j$ in \eqref{energy level sum}, it holds

\begin{align}\label{energy level inequality}
I_{\mu}^{\infty}(v_j)&\geq c_{\mu}^{\infty} \nonumber \\
&=\inf_{\gamma \in \Gamma^{\infty}} \max_{t\in [0,1]}  I_{\mu}^{\infty}(\gamma(t)),\nonumber \\
&\geq \inf_{\gamma \in \Gamma^{\infty}} \max_{t\in [0,1]}  I_{\mu}(\gamma(t)) \\
&\geq \inf_{\gamma \in \Gamma} \max_{t\in [0,1]}  I_{\mu}(\gamma(t)) \nonumber \\
&=c_{\mu}.\nonumber\\
\nonumber
\end{align}
\noindent We now assume, by contradiction, $v_0 \equiv 0$ in \eqref{energy level sum}, which would imply $I_{\mu}(v_0)=0$. Using this and \eqref{energy level inequality}, we see from \eqref{energy level sum} that there exists one nontrivial $v_j$, call it $v_1$, such that $v_1$ is a nonnegative solution of \eqref{SP one equation perturbed infinity} and

\begin{equation}\label{energy level equal} 
I_{\mu}^{\infty}(v_1)=c_{\mu}^{\infty}= c_{\mu}.
\end{equation}\\
\noindent Define $\bar{v}_t(x) =t^2 v_1(tx)$ and $\gamma: \R \to H^1(\R^3)$, $\gamma(t)=\bar{v}_t$. By Lemma 3.3 in \cite[p.\ 663]{RuizJFA}, the function $f(t)=I^\infty_\mu(\bar{v}_t)$ has a unique critical point corresponding to its maximum, and it can be shown that $f'(1)=0$ by Nehari's and Pohozaev's identities for $v_1.$ We deduce that 

\[\max_{t\in\R} I_{\mu}^{\infty}(\gamma(t))=I_{\mu}^{\infty} (v_1),\]\\
and that there exists $M>0$ such that

\[I^{\infty}_{\frac{1}{2}}(\gamma(M))<0,\]
and

\[\max_{t\in\R} I_{\mu}^{\infty}(\gamma(t))= \max_{t\in[0,M]} I_{\mu}^{\infty}(\gamma(t)).\]\\
We then define $\gamma_0:[0,1] \to H^1(\R^3)$, $\gamma_0(t)=\gamma (Mt)$, and see from the above work that $\gamma_0 \in \Gamma^{\infty}$. Therefore, we have that

\begin{align*}
I_{\mu}^{\infty} (v_1)&=\max_{t\in\R} I_{\mu}^{\infty}(\gamma(t))\\
&=\max_{t\in[0,M]} I_{\mu}^{\infty}(\gamma(t))\\
&=\max_{t\in[0,1]} I_{\mu}^{\infty}(\gamma_0(t)).\\
\end{align*}
\noindent Since we have $v_1>0$ on $B$ where $\rho(x) < \rho_{\infty}$ by Lemma \ref{lower bound nontriv sols infinity}, it follows that 

\begin{align*}
c_{\mu}^{\infty}&= I_{\mu}^{\infty}(v_1)\\
&=\max_{t\in[0,1]} I_{\mu}^{\infty}(\gamma_0(t))\\
&>\max_{t\in[0,1]} I_{\mu}(\gamma_0(t))\\
&\geq \inf_{\gamma \in \Gamma^{\infty}} \max_{t\in[0,1]} I_{\mu}(\gamma(t))\\
&\geq \inf_{\gamma \in \Gamma} \max_{t\in[0,1]} I_{\mu}(\gamma(t))\\
&=c_{\mu},\\
\end{align*}
\noindent which contradicts \eqref{energy level equal}. Therefore, we have shown that $v_0\not\equiv 0$. Now, since $v_0$ is a nontrivial and nonnegative solution of \eqref{SP one equation perturbed}, then $I_{\mu}(v_0) > 0$ by Lemma \ref{lower bound energy nontriv sols}. Putting this and \eqref{energy level inequality} together in \eqref{energy level sum}, it follows that $v_j \equiv 0$ for all $1\leq j \leq l$. Therefore, $u_n \to {v_0}$ in $H^1$ by $(iv)$ of Proposition \ref{splitting theorem}, $I_{\mu}(v_0)=c_{\mu}$ by \eqref{energy level sum}, and $I_{\mu}'(v_0)=0$ since $v_0$ is a nonnegative solution of \eqref{SP one equation perturbed}. This concludes the proof of Claim 1.\\

\noindent \textbf{Claim 2.}
\emph{Let $\mu_n \to 1$ be an increasing sequence in $(1-\epsilon, 1] \cap \mathcal{M}$ and (resp.) $\mathcal{M}$ under assumptions $(i)$ and (resp.) $(ii)$. Assume $(u_n)_{n\in\N}\subset H^1$ is such that $u_n$ is nonnegative, $I_{\mu_n}(u_n)=c_{\mu_n}$ and $I'_{\mu_n}(u_n)=0$ for each $n$. Then, there exists a nonnegative $u\in H^1$ such that, up to a subsequence, $u_n\rightarrow u$ in $H^1$, $I(u)=c$, and $I'(u)=0$.}\\

Since $u_n$ is nonnegative, $(u_n)_+=u_n$, and so we can see that

\begin{align}\label{I(u_n) equation}
I(u_n) &= I_{\mu_n}(u_n)+\frac{\mu_n-1}{p+1} \int_{\R^3}u_n^{p+1} \nonumber \\
&=c_{\mu_n} +\frac{\mu_n-1}{p+1} \int_{\R^3}u_n^{p+1},\\
\nonumber
\end{align}
\noindent and, for all $v\in H^1(\R^3)$,

\begin{align}\label{I'(u_n) equation}
I'(u_n)(v)&=I'_{\mu_n}(u_n)(v)+(\mu_n-1)\int_{\R^3}u_n^{p} v \nonumber \\
&\leq |\mu_n-1| \,||u_n||^{p}_{L^{p+1}}||v||_{L^{p+1}} \nonumber \\
&\leq S_{p+1}^{-\frac{1}{2}} |\mu_n-1| \,||u_n||^{p}_{L^{p+1}}||v||_{H^1}.\\
\nonumber
\end{align}
\noindent Set $\alpha_n=\int_{\R^3} (|\nabla u_n|^2+ u_n^2)$, $\gamma_n=\int_{\R^3} \rho\phi_{u_n} u_n^2 $, and $\delta_n=\mu_n \int_{\R^3} u_n^{p+1}$. As in Theorem \ref{theorem existence coercive} Claim 2, we see that $\alpha_n$, $\gamma_n$, and $\delta_n$ satisfy \eqref{mainSOE}, and thus we can obtain that $\alpha_n$, $\gamma_n$, and $\delta_n$ are all bounded. Therefore, using this, \eqref{I(u_n) equation}, \eqref{I'(u_n) equation}, and the fact that $c_{\mu_n} \to c$ as $ \mu_n \nearrow 1$ by definition \eqref{minmax level0}, we can deduce that $||u_n||_{H^1}$ is bounded, $I(u_n) \to c$ and $I'(u_n) \to 0$ as $n\to +\infty$. That is, we have shown that $(u_n)_{n\in\N}$ is a bounded Palais-Smale sequence for $I=I_1$ at the level $c=c_1$, and so, $1\in\mathcal{M}$. By Claim 1, it follows that  there exists a nonnegative ${u} \in H^1$ such that, up to a subsequence, $u_n \to {u}$ in $H^1$, $I({u}) = c$, and $I'({u})= 0$. \\


\noindent {\bf Conclusion.} \noindent
Let $\mu_n \to 1$ be an increasing sequence in $(1-\epsilon, 1] \cap \mathcal{M}$ and (resp.) $\mathcal{M}$ under assumptions $(i)$ and (resp.) $(ii)$. By Claim $1$, we can choose $(u_n)_{n\in\N}\subset H^1$ such that $u_n$ is nonnegative, $I_{\mu_n}(u_n)=c_{\mu_n}$ and $I'_{\mu_n}(u_n)=0$ for each $n$. By Claim $2$, it follows that there exists a nonnegative $u\in H^1$ such that, up to a subsequence, $u_n\rightarrow u$ in $H^1$, $I(u)=c$, and $I'(u)=0$. That is, we have shown $(u, \phi_u) \in H^1(\R^3) \times D^{1,2}(\R^3)$ solves \eqref{main SP system}. Since $u$ and $\phi_u$ are nonnegative by construction, by regularity and the strong maximum principle, it follows that they are, in fact, strictly positive. This concludes the proof. \end{proof}

\subsection{Proof of Theorem \ref{noncoercive large p}}

\begin{proposition}\label{PS cond} [{\bf Palais-Smale condition for $p\geq 3$}] Let $\rho\in C(\R^3)$ be nonnegative such that $\rho(x)\rightarrow \rho_\infty>0$ as $|x|\rightarrow +\infty$ and suppose either of the
following conditions hold:\\

$(i)$ $c<c^{\infty}$ \\

\noindent or\\

$(ii)$ $\rho(x) \leq \rho_{\infty}$ for all $x\in \R^3$, with strict inequality, $\rho(x) < \rho_{\infty}$, on some ball $B\subset \R^3,$\\

\noindent where $c$ and (resp.) $c^\infty$ are defined in (\ref{minmax level0}) and (resp.) (\ref{minmax level00}).
Then, for any $p\in[3,5)$, every Palais-Smale sequence $(u_n)_{n\in\N}\subset H^{1}(\R^3)$ for $I$, at the level $c$, is relatively compact. In particular, $c$ is 
a critical level for $I.$
\end{proposition}

\begin{proof}
\noindent Since, for $p\geq 3$, we have

\[c+1+o(1)\|u_n\|_{H^1(\R^3)}\geq (p+1)I(u_n)-I'(u_n)u_n\geq \|u_n\|^2_{H^1(\R^3)},\]\\
it follows that $(u_n)_{n\in\N}$ is bounded.
By the definition of $u_n$ and Proposition \ref{splitting theorem} with $\mu=1$, it holds that

\begin{equation}\label{energy level sum0}
c =I(v_0)+\sum_{j=1}^{l}I^{\infty}(v_j),\\
\end{equation}\\
where $v_0$ is a nonnegative solution of \eqref{SP one equation} and $v_j$ are nonnegative solutions of \eqref{SP infinity} for $1\leq j \leq l$. Reasoning as in Claim 1 of Theorem \ref{general theorem}, setting $\mu=1$ and replacing $c_{\mu}$, $c_{\mu}^{\infty}$, $\Gamma$, and $\Gamma^{\infty}$ with $c$, $c^{\infty}$, $\bar{\Gamma}$, and $\bar{\Gamma}^{\infty}$, respectively, throughout, the statement follows. This concludes the proof.
\end{proof}

\begin{proof}[Proof of Theorem \ref{noncoercive large p}]
The regularity and the strong maximum principle imply that the nontrivial and nonnegative critical point, $u$, of $I$, found in Proposition \ref{PS cond}, is strictly positive. For the same reason, $\phi_u >0$ everywhere.
\end{proof}

\subsection{Proof of Corollary \ref{theorem existence groundstate noncoercive}}

\begin{proof}[Proof of Corollary \ref{theorem existence groundstate noncoercive}]
If $p\in(3,5)$, we can use the Nehari characterisation of the mountain pass level \eqref{Nehari} with $E=H^1$ to see that the mountain pass solution $u$ found in Theorem \ref{noncoercive large p} is a least energy solution for $I$. If $p\in(2,3]$, we set 

\[c^*\coloneqq \inf_{u\in \mathcal{A}} I(u),\text{ where } \mathcal{A}\coloneqq\{u\in H^1(\R^3)\setminus \{0\} : u \text{ is a nonnegative solution to } \eqref{SP one equation}\},\] \\
and can show that $\mathcal{A}$ is nonempty and $c^*$ is well-defined using the mountain pass critical points that we found in Theorem \ref{noncoercive large p} and Theorem \ref{general theorem} when $p=3$ and $p\in(2,3)$, respectively. It is important to note that when $p=3$, the critical point, $u\in\mathcal{A}$, that we found in Theorem \ref{noncoercive large p} satisfies $I(u)=c$, which implies $c^*\leq c$. Similarly, when $p\in(2,3)$, we can show $c^*\leq c$ using the critical point that we found in Theorem \ref{general theorem}. 
Now, for any $p\in(2,3)$, arguing as in the proof of Corollary \ref{theorem existence groundstate coercive}, we can show that there exists a bounded sequence $(w_n)_{n\in\N}\subset \mathcal{A}$ such that $I(w_n)\to c^*$ as $n\to+\infty$ and $I'(w_n)=0$. By applying Proposition \ref{splitting theorem} with $\mu=1$ to $(w_n)_{n\in\N}$, we can see that

\begin{equation*}
c\geq c^* =I(v_0)+\sum_{j=1}^{l}I^{\infty}(v_j),\\
\end{equation*}\\
where $v_0$ is a nonnegative solution of \eqref{SP one equation} and $v_j$ are nonnegative solutions of \eqref{SP infinity} for $1\leq j \leq l$. Reasoning as in Claim 1 of Theorem \ref{general theorem} with $\mu=1$ and replacing $c_{\mu}$, $c_{\mu}^{\infty}$, $\Gamma$, and $\Gamma^{\infty}$ with $c$, $c^{\infty}$, $\bar{\Gamma}$, and $\bar{\Gamma}^{\infty}$, respectively, throughout, we can show $I(v_0)=c^*$ and $I'(v_0)=0$. We note that by Lemma \ref{lower bound energy nontriv sols}, it holds that $c^* \geq C >0$ for some uniform constant $C>0$, and so it follows that $v_0$ is a nontrivial least energy critical point of $I$. The strict positivity of $v_0$ and $\phi_{v_0}$ follows by regularity and the strong maximum principle since they are nonnegative by construction. This concludes the proof.
\end{proof}

\section{Necessary conditions for concentration of semiclassical states}\label{Necessary conditions for concentration of semiclassical states}
\subsection{Proof of Theorem \ref{necessary conditions E}}

We first prove a necessary condition for the concentration of positive solutions in the semiclassical limit, $\epsilon\to0^+$, in $E$.

\begin{proof}[Proof of Theorem \ref{necessary conditions E}]
We will break the proof into five claims.\\

\noindent \textbf{Claim 1.}
\emph{$\sup_{\epsilon>0}||u_{\epsilon}||_{L^{\infty}(\R^3)} < +\infty$}\\

We will argue by contradiction. Assume, to the contrary, that there exists a sequence $(\epsilon_m)_{m\in\N}$ such that $\epsilon_m \to 0$ as $m\to +\infty$, $u_m \coloneqq u_{\epsilon_m}$ solves \eqref{SPwithlambda} for each $m$, and it holds 

\[||u_m||_{L^{\infty}(\R^3)} \to +\infty \text{ as } m \to +\infty. \]\\
Let

\[ \alpha_m \coloneqq \max u_m, \qquad (\alpha_m \to +\infty \text{ as } m\to +\infty),\]
\[ \beta_m \coloneqq \alpha_m ^{-(p-1)/2}, \qquad (\beta_m \to 0 \text{ as } m\to +\infty).\]\\
Define

\[v_m(x)\coloneqq \frac{1}{\alpha_m} u_m(x_m +\epsilon_m\beta_mx),\]\\
where $x_m$ is a global maximum point of $u_m$. We note that such a point exists because, by  regularity theory, $u_m$ are solutions in the classical sense and, moreover, by the concentration assumption, $u_m$ decays to zero uniformly with respect to $m$. Now, multiplying  \eqref{SPwithlambda} by $\frac{\beta_m^2}{\alpha_m}$, we obtain

\begin{align*}
-\frac{\epsilon_m^2\beta_m^2}{\alpha_m}&\Delta u_m(x_m +\epsilon_m\beta_mx)+\frac{\beta_m^2}{\alpha_m}\lambda u_m(x_m +\epsilon_m\beta_mx)\\
&+\beta_m^2\rho(x_m +\epsilon_m\beta_mx)\phi_{u_m}(x_m +\epsilon_m\beta_mx)\frac{1}{\alpha_m}u_m(x_m +\epsilon_m\beta_mx)=\frac{\beta_m^2}{\alpha_m}u_m^p(x_m +\epsilon_m\beta_mx).\\
\end{align*}
\noindent Noting that $\Delta v_m(x)=\epsilon_m^2\beta_m^2\Delta u_m(x_m +\epsilon_m\beta_mx)/\alpha_m$ and $\beta_m^2/\alpha_m=1/\alpha_m^p$, we see that $v_m$ satisfies

\begin{equation*}
-\Delta v_m + \beta_m^2 \lambda v_m + \beta_m^2\rho (x_m +\epsilon_m\beta_mx)\phi_{u_m}(x_m +\epsilon_m\beta_mx)v_m= v_m^p.\\
\end{equation*}\\
\noindent We further note that 

\begin{align*}
\phi_{u_m}(x_m +\epsilon_m\beta_mx) &= \int_{\R^3} \frac{u_m^2(y)\rho(y)}{4\pi |x_m+\epsilon_m \beta_m x -y| }\dif y\\
&=\int_{\R^3} \frac{u_m^2(x_m+\epsilon_m \beta_m y)\rho(x_m+\epsilon_m \beta_m y)}{4\pi |x_m+\epsilon_m \beta_m x -x_m-\epsilon_m \beta_m y| }\cdot \epsilon_m^3 \beta_m^3 \dif y\\
&=\epsilon_m^2 \beta_m^2\alpha_m^2 \int_{\R^3} \frac{v_m^2( y)\rho(x_m+\epsilon_m \beta_m y)}{4\pi | x - y| }\dif y,\\
\end{align*}
where we have used the change of variables $y \to x_m+\epsilon_m \beta_m y$ in going from the first to second line. Therefore, $v_m$ satisfies

\begin{align*}
-\Delta v_m + \beta_m^2 \lambda v_m + \beta_m^2 \rho (x_m +\epsilon_m\beta_mx)\left( \epsilon_m^2 \beta_m^2\alpha_m^2 \int_{\R^3} \frac{v_m^2( y)\rho(x_m+\epsilon_m \beta_m y)}{4\pi| x - y| } \dif y\right)v_m= v_m^p.\\
\end{align*}
\noindent Namely, since $\beta_m^4\alpha_m^2 =\alpha_m^{-2(p-1)}\alpha_m^2 =\alpha_m^{4-2p}$ (by the definition of $\beta$), we have that $v_m$ satisfies 

\begin{align}\label{vmequation}
-\Delta v_m =- \beta_m^2 \lambda v_m - \epsilon_m^2  \alpha_m^{4-2p}\rho (x_m +\epsilon_m\beta_mx) \left( \int_{\R^3} \frac{v_m^2( y)\rho(x_m+\epsilon_m \beta_m y)}{4\pi| x - y| }\dif y\right)v_m + v_m^p.\\
\nonumber
\end{align}
\noindent It is worth noting here that since $\alpha_m \to +\infty$ as $m\to +\infty$, then $\alpha_m^{4-2p} \to 0$ as $m\to +\infty$ for $p>2$ and $\alpha_m^{4-2p}\to 1$ as $m\to +\infty$ for $p=2$ \footnote{This is the only point in which we use the restriction $p\geq2$.}. \\

We now fix some compact set $K$. We notice that, by construction, $||v_m||_{L^{\infty}(\R^3)} =1$ for all $m$, and, by assumption, $\rho$ is continuous. We also highlight that due to the concentration assumption, we have that the sequence of global maximum points $x_m$ is uniformly bounded with respect to $m$. So, since $v_m^2\rho$ is uniformly bounded in $L^{\infty}(K)$, then $\int_{\R^3} \frac{v_m^2( y)\rho(x_m+\epsilon_m \beta_m y)}{4\pi | x - y| }\dif y$ is uniformly bounded in $C^{0,\alpha}(K)$ and consequently, is uniformly bounded in $L^{\infty}(K)$ (see e.g.\ \cite[p.\ 260]{Lieb and Loss}; \cite[p.\ 11]{Adams}). Thus, the entire right-hand side of \eqref{vmequation} is uniformly bounded in $L^{\infty}(K)$ which implies $v_m$ is uniformly bounded in $C^{1,\alpha}(K)$ (see e.g.\ \cite{Gilbarg and Trudinger}). It then follows that the right-hand side of \eqref{vmequation} is uniformly bounded in $C^{0,\alpha}(K)$, and therefore $v_m$ is uniformly bounded in $C^{2,\alpha}(K)$ by Schauder estimates (see e.g.\ \cite{Gilbarg and Trudinger}). Namely, for $x,y\in K$, $x \neq y$, and for every $m$, it holds that 

\[ |\partial^{\beta} v_m(x) - \partial^{\beta}v_m(y)| \leq C_K|x-y|^{\alpha},  \quad |\beta|\leq 2,\]\\
\noindent for some constant $C_K$ which depends on $K$ but does not depend on $m$. It follows that uniformly on compact sets and for some $v_0\in C^2(\R^3)$,

\[\partial^{\beta} v_m \to \partial^{\beta} v_0  \text{ as } m\to +\infty, \quad |\beta|\leq 2.\]\\
Therefore, taking the limit $m \to +\infty$ in \eqref{vmequation} we get 

\begin{equation*}
\left\{
\begin{array}{lll}
  - \Delta v_0 = v_0^{p},   \quad &x\in \R^3 \\
  \,\,\, v_0(0)=1,
\end{array}
\right.
\end{equation*}\\
\noindent where the second equality has come from the fact that $v_m(0)=u_m(x_m)/\alpha_m=\alpha_m/\alpha_m=1$ for all $m$. On the other hand from the equation, by a celebrated result of Gidas-Spruck \cite{Gidas and Spruck} we infer $v_0 \equiv 0.$  So, we have reached a contradiction, and thus $\sup_{\epsilon>0}||u_{\epsilon}||_{L^{\infty}(\R^3)} < +\infty$. \\

\noindent \textbf{Claim 2.} \emph{Assume there exists a sequence $(\epsilon_k)_{k\in\N}$ such that $\epsilon_k \to 0$ as $k\to +\infty$ and $u_k \coloneqq u_{\epsilon_k}$ solves \eqref{SPwithlambda} for each $k$. Let $w_k(x)\coloneqq u_k(x_0 +\epsilon_kx)$, where $x_0$ is a concentration point for $u_k$. Then, 
\begin{enumerate}[(i)]
\item up to a subsequence, $w_k\to$ some $w_0$ in $C^2_{\textrm{loc}
}(\R^3)$,
\item $w_0 > 0$.\\
\end{enumerate}}

We begin by proving $(i)$. We first notice that $w_k$ solves 

\begin{equation}\label{rescaledSPinE}
\left\{
\begin{array}{lll}
  -\Delta w_k+ \lambda w_k +  \rho (x_0 +\epsilon_kx) \phi_{u_k}(x_0 +\epsilon_kx) w_k = w_k^p,   \quad &x\in \R^3 \\
  \,\,\, -\Delta \phi_{u_k}(x_0 +\epsilon_kx)=\rho(x_0 +\epsilon_kx)w_k^2, &x\in \R^3.\\
\end{array}
\right.
\end{equation}\\
\noindent We note that

\begin{align*}
\phi_{u_k}(x_0 +\epsilon_kx) &= \int_{\R^3} \frac{u_k^2(y)\rho(y)}{4\pi|x_0+\epsilon_k x -y| }\dif y\\
&=\int_{\R^3} \frac{u_k^2(x_0+\epsilon_k y)\rho(x_0+\epsilon_k y)}{4\pi|x_0+\epsilon_k x -x_0-\epsilon_k y| }\cdot \epsilon_k^3 \dif y\\
&=\epsilon_k^2  \int_{\R^3} \frac{w_k^2( y)\rho(x_0+\epsilon_k y)}{4\pi| x - y| }\dif y,\\
\end{align*}
where we have used the change of variables $y \to x_0+\epsilon_k y$ in going from the first to second line. So, $w_k$ solves 

\begin{equation}\label{wkequationE}
-\Delta w_k= - \lambda w_k -\rho (x_0 +\epsilon_kx) \left( \epsilon_k^2  \int_{\R^3} \frac{w_k^2( y)\rho(x_0+\epsilon_k y)}{4\pi| x - y| }\dif y \right)w_k + w_k^p.
\end{equation}\\
\noindent We now once again fix some compact set $K$. We notice that, by Claim 1, $\sup_{k>0}||w_k||_{L^{\infty}(\R^3)} <+\infty$, and, by assumption, $\rho$ is continuous. So, since $w_k^2\rho$ is uniformly bounded in $L^{\infty}(K)$, then $\int_{\R^3} \frac{w_k^2( y)\rho(x_0+\epsilon_k y)}{4\pi| x - y| }\dif y$ is uniformly bounded in $C^{0,\alpha}(K)$ and thus, is uniformly bounded in $L^{\infty}(K)$ (see e.g.\ \cite[p.\ 260]{Lieb and Loss}; \cite[p.\ 11]{Adams}). Therefore, the right-hand side of \eqref{wkequationE} is uniformly bounded in $L^{\infty}(K)$ which implies $w_k$ is uniformly bounded in $C^{1,\alpha}(K)$ (see e.g.\ \cite{Gilbarg and Trudinger}). It follows that the right-hand side of \eqref{wkequationE} is uniformly bounded in $C^{0,\alpha}(K)$, and thus, by Schauder estimates, we have that $w_k$ is uniformly bounded in $C^{2,\alpha}(K)$ (see e.g.\ \cite{Gilbarg and Trudinger}). Since this holds for every compact set contained in $\R^3$, arguing the same way as in Claim 1, it follows that uniformly on compact sets and for some $w_0\in C^2(\R^3)$,

\[\partial^{\beta} w_k \to \partial^{\beta} w_0  \text{ as } k\to +\infty, \quad |\beta|\leq 2.\]\\
\noindent Therefore, taking the limit $k\to +\infty$ in \eqref{wkequationE}, we have 

\begin{equation}\label{w0equation}
-\Delta w_0 +\lambda w_0 =w_0^p, \quad x\in \R^3.
\end{equation}\\

We now aim to prove $(ii)$. Let $x_k$ be a maximum point of $u_k$. Since $u_k$ is a solution to \eqref{SPwithlambda}, we have that 

\[-\epsilon_k^2 \Delta u_k(x_k)+ \lambda u_k(x_k) +\rho (x_k) \phi_{u_{x_k}}(x_k)  u_k(x_k) = u_k^{p} (x_k).\]\\
Noting that $\Delta u_k(x_k) \leq 0$ since $x_k$ is a maximum point of $u_k$, we see that

\[[\lambda +\rho (x_k)\phi_{u_{x_k}}(x_k) ]u_k(x_k)\leq  u_k^{p} (x_k),\]\\
and so

\begin{equation}\label{max values u_k}
u_k(x_k)\geq [\lambda+\rho (x_k)\phi_{u_{x_k}}(x_k) ]^{\frac{1}{p-1}} \geq \lambda^{\frac{1}{p-1}}>0.\\
\end{equation}\\
Therefore, the local maximum values of $u_k$, and hence of $w_k$, are greater than or equal to $\lambda^{\frac{1}{p-1}}$, and since $w_k \to w_0$ in $C^{2}_{\textrm{loc}
}(\R^3)$, then $w_0 \not\equiv 0$. In particular, this and \eqref{w0equation}, imply that $w_0>0$ by the strong maximum principle. \\

\noindent \textbf{Claim 3.} \emph{For large $k$, it holds that $\int_{\R^3}\int_{\R^3}\frac { w_k^2(y) \rho (x_0 +\epsilon_ky)w_k^2(x)\nabla\rho (x_0 +\epsilon_kx)}{4\pi |x-y|} \dif y \,\dif x=0.$}\\\\

We first recall that $w_k$, as defined in Claim 2, solves \eqref{rescaledSPinE}. Multiplying the first equation in \eqref{rescaledSPinE} by $\nabla w_k$ and integrating on $B_R(0)$, we get

\begin{align*}
0=&\int_{B_R} \Delta w_k \nabla w_k \dif x - \int_{B_R} \lambda\frac{\nabla w_k^2}{2}\dif x - \frac{1}{2}\int_{B_R} \nabla (\rho (x_0 +\epsilon_kx)\phi_{u_k}(x_0 +\epsilon_kx)  w_k^2)\dif x\\
&+\frac{\epsilon_k}{2}\int_{B_R} \rho (x_0 +\epsilon_kx) \nabla\phi_{u_k}(x_0 +\epsilon_kx)w_k^2\dif x +\frac{\epsilon_k}{2}\int_{B_R}\nabla\rho (x_0 +\epsilon_kx)\phi_{u_k}(x_0 +\epsilon_kx)w_k^2\dif x\\
&+\int_{B_R} \frac{\nabla w_k^{p+1}}{p+1}\dif x.\\
\end{align*}
\noindent By using the divergence theorem and rearranging terms, this becomes

\begin{align}\label{bigequation}
\frac{\epsilon_k}{2}\int_{B_R}&\nabla\rho (x_0 +\epsilon_kx)\phi_{u_k}(x_0 +\epsilon_kx)w_k^2\dif x\nonumber \\
&= \int_{\partial B_R} \left( \lambda \frac{w_k^2}{2}\nu -\frac{w_k^{p+1}}{p+1}\nu + \frac{1}{2} \rho (x_0 +\epsilon_kx) \phi_{u_k}(x_0 +\epsilon_kx) w_k^2 \nu \right) \dif \sigma \\
& \qquad -\frac{\epsilon_k}{2}\int_{B_R}\ \rho (x_0 +\epsilon_kx) \nabla\phi_{u_k}(x_0 +\epsilon_kx) w_k^2\dif x-\int_{B_R} \Delta w_k \nabla w_k \dif x,\nonumber\\
\nonumber
\end{align}
where $\nu$ is the exterior normal field on $B_R$. We now focus on the second integral on the right-hand side of this equality. We begin by noting that if we multiply the second equation in \eqref{rescaledSPinE} by $\nabla \phi_{u_k}(x_0 +\epsilon_kx)$ and integrate on $B_R(0)$, we get 

\begin{equation}\label{Poisson eq}
-\int_{B_R} \rho (x_0 +\epsilon_kx) \nabla\phi_{u_k}(x_0 +\epsilon_kx)w_k^2\dif x = \int_{B_R}\Delta\phi_{u_k}(x_0 +\epsilon_kx)  \nabla\phi_{u_k}(x_0 +\epsilon_kx) \dif x.\\
\end{equation}\\
\noindent Moreover, using the divergence theorem, we see that
\pagebreak

\begin{align}\label{Poisson after div theorem}
\frac{\epsilon_k}{2}\int_{B_R}\Delta\phi_{u_k}(x_0 +\epsilon_kx)  \frac{\partial}{\partial x_i}\phi_{u_k}(x_0 +\epsilon_kx) dx &= \frac{1}{2}\int_{B_R} \text{div} \left(\nabla \phi_{u_k}(x_0 +\epsilon_kx)  \frac{\partial}{\partial x_i}\phi_{u_k}(x_0 +\epsilon_kx) \right)\dif x \nonumber\\
&\quad -\frac{1}{2}\int_{B_R} \nabla \phi_{u_k}(x_0 +\epsilon_kx)  \frac{\partial}{\partial x_i}(\nabla\phi_{u_k}(x_0 +\epsilon_kx))\dif x \nonumber \\
&= \frac{1}{2} \int_{\partial B_R}\bigg( \frac{\partial \phi_{u_k}(x_0 +\epsilon_kx)}{\partial \nu}\frac{\partial}{\partial x_i} \phi_{u_k}(x_0 +\epsilon_kx) \\
&\quad -\frac{1}{2}|\nabla \phi_{u_k}(x_0 +\epsilon_kx)|^2\nu_i\bigg)\dif \sigma.\nonumber \\
\nonumber
\end{align}
\noindent Therefore, combining \eqref{Poisson eq} and \eqref{Poisson after div theorem}, we obtain 

\begin{align}\label{secondintegral}
-\frac{\epsilon_k}{2}\int_{B_R} \rho (x_0 +\epsilon_kx)&\nabla\phi_{u_k}(x_0 +\epsilon_kx) w_k^2\dif x\nonumber\\
&=\frac{1}{2} \int_{\partial B_R}\bigg( \nabla \phi_{u_k}(x_0 +\epsilon_kx)\frac{\partial \phi_{u_k}(x_0 +\epsilon_kx)}{\partial \nu} -\frac{1}{2}|\nabla \phi_{u_k}(x_0 +\epsilon_kx)|^2\nu\bigg)\dif \sigma.\\
\nonumber
\end{align}
\noindent Turning our attention to the third integral on the right-hand side of \eqref{bigequation} and by arguing in a similar way as above, we can show that

\begin{equation}\label{thirdintegral}
\int_{B_R}\Delta w_k \nabla w_k \dif x = \int_{\partial B_R}\left( \nabla w_k\frac{\partial w_k}{\partial \nu} -\frac{1}{2}|\nabla w_k|^2\nu\right)\dif \sigma.\\
\end{equation}\\
\noindent Therefore, using \eqref{secondintegral} and \eqref{thirdintegral}, we see that \eqref{bigequation} becomes

\begin{align}\label{surfaceintegralequationtheorem2}
\frac{\epsilon_k}{2}\int_{B_R}\nabla\rho (x_0 +\epsilon_kx)&\phi_{u_k}(x_0 +\epsilon_kx)w_k^2\dif x \nonumber\\
& = \int_{\partial B_R} \bigg(\lambda \frac{w_k^2}{2}\nu -\frac{w_k^{p+1}}{p+1}\nu + \frac{1}{2} \rho (x_0 +\epsilon_kx)\phi_{u_k}(x_0 +\epsilon_kx)  w_k^2 \nu  \\
&+\frac{1}{2} \nabla \phi_{u_k}(x_0 +\epsilon_kx)\frac{\partial \phi_{u_k}(x_0 +\epsilon_kx)}{\partial \nu} -\frac{1}{4}|\nabla \phi_{u_k}(x_0 +\epsilon_kx)|^2\nu \nonumber\\
& - \nabla w_k\frac{\partial w_k}{\partial \nu} +\frac{1}{2}|\nabla w_k|^2\nu \bigg)\dif \sigma.\nonumber\\
\nonumber
\end{align}
\noindent Call the integral on the right-hand side of this equation $I_R$. Then,

\begin{align*}
|I_R| &\leq \int_{\partial B_R} \bigg( \lambda \frac{w_k^2}{2} +\frac{w_k^{p+1}}{p+1} + \frac{1}{2}\rho (x_0 +\epsilon_kx) \phi_{u_k}(x_0 +\epsilon_kx)  w_k^2 +\frac{1}{2} |\nabla \phi_{u_k}(x_0 +\epsilon_kx)|^2\\
&\quad + \frac{1}{4}|\nabla \phi_{u_k}(x_0 +\epsilon_kx)|^2 +|\nabla w_k|^2 +\frac{1}{2}|\nabla w_k|^2 \bigg)\dif \sigma\\
& \leq \frac{3}{2} \int_{\partial B_R} \bigg( \lambda w_k^2 +w_k^{p+1}+ \rho (x_0 +\epsilon_kx) \phi_{u_k}(x_0 +\epsilon_kx) w_k^2 + |\nabla \phi_{u_k}(x_0 +\epsilon_kx)|^2 +|\nabla w_k|^2\bigg)\dif \sigma.\\
\end{align*}
\noindent So, 

\begin{align*}
\int_0^{+\infty} |I_R| & \leq \int_0^{+\infty} \frac{3}{2} \int_{\partial B_R} \bigg( \lambda w_k^2 +w_k^{p+1}+  \rho (x_0 +\epsilon_kx)\phi_{u_k}(x_0 +\epsilon_kx) w_k^2  \\
& \qquad + |\nabla \phi_{u_k}(x_0 +\epsilon_kx)|^2 +|\nabla w_k|^2\bigg)\dif \sigma \dif R\\
& =\frac{3}{2} \int_{\R^3} \bigg( \lambda w_k^2 +w_k^{p+1}+ \rho (x_0 +\epsilon_kx) \phi_{u_k}(x_0 +\epsilon_kx) w_k^2 \\
&\qquad+ |\nabla \phi_{u_k}(x_0 +\epsilon_kx)|^2  +|\nabla w_k|^2\bigg)\dif x\\
\\
&<+\infty \text{ for each } k,\\
\end{align*}
\noindent since $w_k$ is a solution to \eqref{rescaledSPinE}. Thus, for each fixed $k$, there exists a sequence $R_m \to +\infty$ as $m\to +\infty$ such that $I_{R_m} \to 0$ as $m\to +\infty$. Letting $R=R_m \to +\infty$ in \eqref{surfaceintegralequationtheorem2} yields

\begin{align*}
0&=\frac{\epsilon_k}{2}\int_{\R^3}\nabla\rho (x_0 +\epsilon_kx)\phi_{u_k}(x_0 +\epsilon_kx)w_k^2\dif x\\
&= \frac{\epsilon_k}{2}\int_{\R^3}\int_{\R^3}\frac {\epsilon_k^2 w_k^2(y) \rho (x_0 +\epsilon_ky)w_k^2(x)\nabla\rho (x_0 +\epsilon_kx)}{4\pi |x-y|} \dif y \,\dif x.\\
\end{align*}
\noindent Since this holds for each fixed $k$, we have 

\begin{equation}\label{final integral E}
\int_{\R^3}\int_{\R^3}\frac { w_k^2(y) \rho (x_0 +\epsilon_ky)w_k^2(x)\nabla\rho (x_0 +\epsilon_kx)}{4\pi |x-y|} \dif y \,\dif x=0.\\
\end{equation}\\

\noindent \textbf{Claim 4.} \emph{There exists $R_0 >0$ and $C>0$ such that, for $k$ sufficiently large, $w_k(x) \leq C|x|^{-1}e^{-\frac{\sqrt{\lambda}}{2}|x|}$ for all $|x|\geq R_0$.}\\

We first note that, by the concentration assumption, it holds that $w_k \to 0$ as $|x|\to +\infty$. Namely, there exists $R_0>0$, $K>0$ such that 

 \begin{equation}\label{wk bound}
 w_k \leq \left(\frac{\lambda}{2}\right)^{\frac{1}{p-1}}, \quad \forall \, |x|\geq R_0, \,\, \forall k\geq K,
 \end{equation}\\
It follows that 

 \begin{equation*}
 w_k^p \leq \frac{\lambda}{2}w_k, \quad \forall \, |x|\geq R_0, \,\, \forall k\geq K,\\
 \end{equation*}\\
\noindent and therefore, since $w_k$ solves \eqref{rescaledSPinE}, we have, for all $ |x|\geq R_0$ and for all $k\geq K$,

\begin{equation}\label{wk inequality}
-\Delta w_k +\lambda w_k \leq -\Delta w_k +(\lambda+ \rho (x_0 +\epsilon_kx)\phi_{u_k}(x_0 +\epsilon_kx) ) w_k = w_k^p \leq \frac{\lambda}{2}w_k.
\end{equation}\\ 
\noindent Namely, it holds that 

\begin{equation}\label{boundDeltawk}
-\Delta w_k \leq -\frac{\lambda}{2}w_k, \quad \forall \, |x|\geq R_0, \,\, \forall k\geq K.
\end{equation}\\
\noindent Now, define

\[\omega(x)\coloneqq C|x|^{-1}e^{-\frac{\sqrt{\lambda}}{2}|x|}, \quad \text{where } C\coloneqq \left(\frac{\lambda}{2}\right)^{\frac{1}{p-1}}R_0\,e^{\frac{\sqrt{\lambda}}{2}R_0},\]\\
\noindent Then, using this definition and \eqref{wk bound}, we see that 

\begin{equation}\label{wklessthanomega}
w_k(x)  \leq \left(\frac{\lambda}{2}\right)^{\frac{1}{p-1}} = \omega(x), \quad \text{for } |x|=R_0, \,\, \forall k\geq K.
\end{equation}\\
\noindent It can also be checked that, 

\begin{equation}\label{delta omega less than omega}
\Delta \omega \leq \frac{\lambda}{4}\omega, \quad \text{for } |x|\neq 0.
\end{equation}\\
\noindent We then define $\bar{\omega}_k(x)\coloneqq w_k(x) -\omega(x)$. By \eqref{wklessthanomega} it holds that 

\begin{equation}\label{baromegabound}
\bar{\omega}_k(x) \leq 0, \quad \text{for } |x|=R_0, \,\, \forall k\geq K.
\end{equation}\\
\noindent Moreover, using \eqref{boundDeltawk} and \eqref{delta omega less than omega}, it holds that 

\begin{align}\label{baromegaequation}
-\Delta \bar{\omega}_k+ \frac{\lambda}{2}\bar{\omega}_k \leq 0, \quad \forall \, |x|\geq R_0, \,\, \forall k\geq K.
\end{align}\\
\noindent and

\begin{equation}\label{baromegadecay}
\lim_{|x|\to+\infty} \bar{\omega}_k(x) =0.
\end{equation}\\
\noindent Thus, by the maximum principle on unbounded domains (see e.g.\ \cite{Berestycki Caffarelli Nirenberg}), it follows that,

\[w_k(x) \leq C|x|^{-1}e^{-\frac{\sqrt{\lambda}}{2}|x|}, \quad \forall \,  |x|\geq R_0,\]\\
\noindent for $k$ sufficiently large.\\

\noindent \textbf{Claim 5.} $\rho(x_0)\nabla \rho (x_0) =0$.\\

We first pick a uniform large constant $C>0$ such that for all $x\in \R^3$\ and large $k$ it holds that

\begin{equation}\label{definition of tilde w}
w_k(x) \leq \tilde w(x):=
  C(1+|x|)^{-1}e^{-\frac{\sqrt{\lambda}}{2}|x|}.
  \end{equation}\\
We now highlight the fact that due to the concentration assumption, from now on, we can take $k$ large enough and a suitable $\epsilon_1>0$ such that 

\[\epsilon_k < \epsilon_1<\min\left\{\epsilon_0, \frac{\sqrt{\lambda}}{b}\right\}, \quad \text{if } b>0,\]\\
and simply 

\[\epsilon_k < \epsilon_0, \quad \text{if } b\leq0,\]\\
where $\epsilon_0>0$ is defined in the statement of the theorem. We assume that $b>0$ as the case $b\leq 0$ is easier and requires only obvious modifications. 
\noindent By the growth assumption on $\rho$, there exists a uniform constant $C_1>0$ such that for all $x\in \R^3$,

\[|\nabla \rho(x_0 +\epsilon_kx)| \leq C_1 (1 +|x|)^ae^{b\epsilon_1|x|} =: g(x).\]\\
By the mean value theorem we have

\[|\rho(x_0+\epsilon_k y)| \leq |\epsilon_ky||\nabla \rho (x_0 +\theta(\epsilon_k y))| + |\rho(x_0)|,\]\\
for some $\theta \in [0,1]$. Combining this with the estimate on $|\nabla\rho(x_0 +\theta(\epsilon_ky))|$, it follows that for some uniform constant $C_2>0$ and for all $y\in \R^3$, 

\begin{equation*}
|\rho(x_0+\epsilon_k y)| \leq C_2 |y|(1+|y|)^ae^{b\epsilon_1|y|}+|\rho(x_0)|=:f(y).
\end{equation*}\\
\noindent Therefore, putting everything together, we have that, for $k$ sufficiently large, 

\begin{equation}\label{dominating func in E}
\left| \frac { w_k^2(y) \rho (x_0 +\epsilon_ky)w_k^2(x)\nabla\rho (x_0 +\epsilon_kx)}{ (x-y)}  \right| \leq \frac{\tilde{w}^2(y) f(y) \tilde{w}^2(x) g(x)}{|x-y|}.
\end{equation}\\
\noindent The right hand side is a uniform $L^1(\R^{6})$ bound. In fact, using for instance the Hardy-Littlewood-Sobolev inequality, we have 

\begin{align}\label{HLS inequality in E}
\left| \int_{\R^3}\int_{\R^3} \frac{\tilde{w}^2(y) f(y) \tilde{w}^2(x) g(x)}{|x-y|} \dif y \,\dif x \right| &\lesssim ||\tilde{w}^2f||_{L^{6/5}(\R^3)}||\tilde{w}^2g||_{L^{6/5}(\R^3)}\nonumber \\
&< +\infty,\\
\nonumber
\end{align}
as the choice of $\epsilon_1$ implies that $\tilde{w}^2f,\, \tilde{w}^2g \in L^{6/5}(\R^3)$. We now let $k\to +\infty$ in \eqref{final integral E}, and note that by \eqref{dominating func in E}, \eqref{HLS inequality in E}, Claim 2, and the assumption that $\rho \in C^1(\R^3)$, we can use the dominated convergence theorem to obtain

\begin{equation*}
\int_{\R^3}\int_{\R^3}\frac { w_0^2(y) \rho (x_0)w_0^2(x)\nabla\rho (x_0)}{4\pi |x-y|} \dif y \,\dif x =0.\\
\end{equation*}\\
\noindent Then, since $w_0>0$ by Claim 2, we have that 
\[\rho(x_0)\nabla \rho (x_0) =0.\]\\
Since $\rho$ is nonnegative, any zero is global minimum, and therefore we have $\nabla \rho (x_0) =0$.
\end{proof}

\subsection{Proof of Theorem \ref{necessary conditions H}}

We follow up Theorem \ref{necessary conditions E} with a similar result on necessary conditions for concentration of solutions in $H^1(\R^3) \times D^{1,2}(\R^3)$.\\

\begin{proof} [Proof of Theorem \ref{necessary conditions H}]

The proof closely follows that of Theorem \ref{necessary conditions E}. We assert that the same five claims as were made in the proof of Theorem \ref{necessary conditions E} hold, and will only highlight the differences in the proofs of these claims. The proof of Claim 1 and Claim 2 follow similarly as in Theorem \ref{necessary conditions E}, however since $\rho$ is both continuous and globally bounded in this case, we do not need to fix a specific compact set $K$ in the regularity arguments, but instead it follows directly that $v_m$ is uniformly bounded in $C^{2,\alpha}_{\textrm{loc}
}(\R^3)$ and $w_k$ is uniformly bounded in $C^{2,\alpha}_{\textrm{loc}
}(\R^3)$. The proof of Claim 3 and Claim 4 follow exactly as in Theorem \ref{necessary conditions E}.  To prove Claim 5, we define the exponentially decaying function $\tilde{w}$ as in \eqref{definition of tilde w} and since $\rho$ and $\nabla \rho$ are bounded, we have, for $k$ sufficiently large, 

\begin{equation*}\
\left| \frac { w_k^2(y) \rho (x_0 +\epsilon_ky)w_k^2(x)\nabla\rho (x_0 +\epsilon_kx)}{(x-y)} \right| \lesssim \frac{\tilde w^2 (y)\tilde w^2 (x)}{|x-y|}\in L^1(\R^6).
\end{equation*}\\
This is enough to conclude the proof as in Theorem \ref{necessary conditions E} using the dominated convergence theorem.
\end{proof}

\section*{Appendix: proof of Lemma \ref{pohozaevlemma}}

\begin{proof}[Proof of Lemma \ref{pohozaevlemma}]

With the regularity remarks of Section \ref{Regularity and positivity} in place, we now multiply the first equation in \eqref{poh system} by $(x, \nabla u)$ and integrate on $B_R(0)$ for some $R>0$. We will compute each integral separately. We first note that by Lemma 3.1 in \cite{D'Aprile and Mugnai} it holds that 

\begin{equation}\label{POH1}
\int_{B_R}-\Delta u (x, \nabla u) \dif x =  -\frac{1}{2}\int_{B_R}|\nabla u|^2 \dif x -\frac{1}{R}\int_{\partial B_R} |(x,\nabla u)|^2 \dif  \sigma +\frac{R}{2}\int_{\partial B_R}|\nabla u|^2 \dif \sigma.
\end{equation}\\
\noindent Fixing $i =1,2,3$, integrating by parts and using the divergence theorem, we then see that,

\begin{align}
\int_{B_R} bu(x_i \partial_i u)\dif x& =b\left[ -\frac{1}{2} \int_{B_R} u^2\dif x +\frac{1}{2}\int_{B_R}\partial_i(u^2x_i)\dif x \right] \nonumber \\ 
&=b\left[ -\frac{1}{2} \int_{B_R} u^2\dif x + \frac{1}{2}\int_{\partial B_R} u^2 \frac{x_i^2}{|x|}\dif \sigma \right].
\nonumber\\
\nonumber
\end{align}
\noindent So, summing over $i$, we get

\begin{equation}\label{POH2}
\int_{B_R} bu(x, \nabla u) \dif x= b\left[-\frac{3}{2}\int_{B_R}u^2\dif x +\frac{R}{2}\int_{\partial B_R}u^2 \dif \sigma \right].
\end{equation}\\
\noindent Again, fixing $i =1,2,3$, integrating by parts and using the divergence theorem, we find that,

\begin{align}
\int_{B_R}c\rho \phi_u u x_i (\partial_i u) \dif x &= c\bigg[ -\frac{1}{2}\int_{B_R} \rho\phi_u u^2 \dif x -\frac{1}{2}\int_{B_R} \phi_u u^2x_i (\partial_i\rho) \dif x -\frac{1}{2}\int_{B_R} \rho u^2  x_i  (\partial_i \phi_u) \dif x \nonumber \\
& \qquad\qquad +\frac{1}{2}\int_{B_R}\partial_i (\rho\phi_u  u^2 x_i)\dif x \bigg] \nonumber\\
&=c\bigg[ -\frac{1}{2}\int_{B_R}\rho\phi_u  u^2 \dif x -\frac{1}{2}\int_{B_R} \phi_u u^2x_i (\partial_i\rho) \dif x -\frac{1}{2}\int_{B_R} \rho u^2  x_i  (\partial_i \phi_u) \dif x \nonumber \\
& \qquad\qquad +\frac{1}{2}\int_{\partial B_R}\rho \phi_u  u^2 \frac{x_i^2}{|x|}\dif \sigma \bigg]. \nonumber\\
\nonumber
\end{align}
\noindent Thus, summing over $i$, we get

\begin{align}\label{POH3}
\int_{B_R}c\rho\phi_u u (x, \nabla u) \dif x &= c\bigg[ -\frac{3}{2}\int_{B_R}\rho\phi_u  u^2 \dif x -\frac{1}{2}\int_{B_R} \phi_u u^2 (x, \nabla \rho) \dif x  \\
& \qquad\qquad -\frac{1}{2}\int_{B_R} \rho u^2  (x, \nabla \phi_u) \dif x+\frac{R}{2}\int_{\partial B_R} \rho\phi_u  u^2 \dif \sigma \bigg]. \nonumber \\
\nonumber
\end{align}
\noindent Finally, once more fixing $i =1,2,3$, integrating by parts and using the divergence theorem, we find that,

\begin{align*}
\int_{B_R}d |u|^{p-1}u (x_i \partial_i u) \dif x = d \left[ \frac{-1}{p+1} \int_{B_R} |u|^{p+1} \dif x+ \frac{1}{p+1} \int_{\partial B_R} |u|^{p+1} \frac{x_i^2}{|x|} \dif \sigma \right],\\
\end{align*}
\noindent and so, summing over $i$, we see that 

\begin{equation}\label{POH4}
\int_{B_R}d|u|^{p-1}u (x, \nabla u) \dif x=d\left[\frac{-3}{p+1}\int_{B_R}|u|^{p+1}\dif x+\frac{R}{p+1}\int_{\partial B_R}|u|^{p+1}\dif \sigma\right].
\end{equation}\\
\noindent Putting \eqref{POH1}, \eqref{POH2}, \eqref{POH3} and \eqref{POH4} together, we see that

\begin{align}\label{prePOH}
& -\frac{1}{2}\int_{B_R}|\nabla u|^2 \dif x -\frac{1}{R}\int_{\partial B_R} |(x,\nabla u)|^2 \dif \sigma +\frac{R}{2}\int_{\partial B_R}|\nabla u|^2 \dif \sigma  +b\bigg[-\frac{3}{2}\int_{B_R}u^2\dif x \nonumber\\
&\quad +\frac{R}{2}\int_{\partial B_R}u^2 \dif \sigma \bigg] + c\bigg[ -\frac{3}{2}\int_{B_R}\rho\phi_u  u^2 \dif x-\frac{1}{2}\int_{B_R} \phi_u u^2 (x, \nabla \rho) \dif x-\frac{1}{2}\int_{B_R} \rho u^2  (x, \nabla \phi_u) \dif x\nonumber \\
& \quad \quad +\frac{R}{2}\int_{\partial B_R} \rho \phi_u  u^2 \dif \sigma \bigg]-d\left[\frac{-3}{p+1}\int_{B_R}|u|^{p+1}\dif x+\frac{R}{p+1}\int_{\partial B_R}|u|^{p+1}\dif \sigma\right] =0.\\
\nonumber
\end{align}
\noindent We now multiply the second equation in \eqref{poh system} by $(x, \nabla \phi_u)$ and integrate on $B_R(0)$ for some $R>0$. Using Lemma 3.1 in \cite{D'Aprile and Mugnai} we see that 

\begin{align*}
\int_{B_R} \rho u^2 (x, \nabla \phi_u)\dif x &= \int_{B_R} -\Delta \phi_u (x, \nabla \phi_u) \dif x \\
&= -\frac{1}{2}\int_{B_R}|\nabla \phi_u|^2 \dif x -\frac{1}{R}\int_{\partial B_R} |(x,\nabla \phi_u)|^2 \dif  \sigma +\frac{R}{2}\int_{\partial B_R}|\nabla \phi_u|^2 \dif \sigma.
\end{align*}\\
\noindent Substituting this into \eqref{prePOH} and rearranging, we get

\begin{align}\label{POHwithboundary}
& -\frac{1}{2}\int_{B_R}|\nabla u|^2 \dif x -\frac{3b}{2}\int_{B_R}u^2\dif x -\frac{3c}{2}\int_{B_R}\rho\phi_u  u^2 \dif x   \nonumber \\
&\qquad\qquad\qquad -\frac{c}{2}\int_{B_R} \phi_u u^2 (x, \nabla \rho) \dif x+\frac{c}{4}\int_{B_R}|\nabla \phi_u|^2\dif x +\frac{3d}{p+1}\int_{B_R}|u|^{p+1}\dif x \nonumber\\
&\qquad=\frac{1}{R}\int_{\partial B_R} |(x,\nabla u)|^2 \dif \sigma -\frac{R}{2}\int_{\partial B_R}|\nabla u|^2 \dif \sigma -\frac{bR}{2}\int_{\partial B_R}u^2 \dif\sigma - \frac{cR}{2}\int_{\partial B_R} \rho\phi_u  u^2 \dif \sigma  \\
&\qquad\qquad\qquad- \frac{c}{2R}\int_{\partial B_R} |(x,\nabla \phi_u)|^2 \dif \sigma +\frac{cR}{4}\int_{\partial B_R}|\nabla \phi_u|^2 \dif \sigma+\frac{dR}{p+1}\int_{\partial B_R}|u|^{p+1}\dif \sigma.\nonumber\\
\nonumber
\end{align}
\noindent We now call the right hand side of this equality $I_R$. We note that $|(x,\nabla u)|\leq R|\nabla u |$ and $|(x,\nabla \phi_u)|\leq R|\nabla \phi_u |$ on $\partial B_R$, so it holds that

\begin{align*}
|I_R| &\leq \frac{3R}{2}\int_{\partial B_R}|\nabla u|^2 \dif \sigma +\frac{bR}{2}\int_{\partial B_R}u^2 \dif\sigma  \\
&\qquad + \frac{cR}{2}\int_{\partial B_R}\rho \phi_u  u^2 \dif \sigma+\frac{3cR}{4}\int_{\partial B_R}|\nabla \phi_u|^2 \dif \sigma+\frac{dR}{p+1}\int_{\partial B_R}|u|^{p+1}\dif\sigma.\\
\nonumber
\end{align*}
\noindent Now, since $|\nabla u|^2$, $u^2 \in L^1(\R^3)$ because $u\in E (\R^3)\subset H^1(\R^3)$, $ \rho\phi_u u^2$, $|\nabla \phi_u|^2 \in L^1(\R^3)$ because $\int_{\R^3}\rho\phi_u u^2\dif x=\int_{\R^3}|\nabla \phi_u|^2\dif x$ and $\phi_u \in D^{1,2}(\R^3)$, and $|u|^{p+1} \in L^1(\R^3)$ because $E(\R^3) \hookrightarrow L^q(\R^3)$ for all $q \in [2,6]$, then it holds that $I_{R_n} \to 0$ as $n\to +\infty$ for a suitable sequence $R_n \to +\infty$ (see e.g. \cite{D'Aprile and Mugnai}). Thus, considering \eqref{POHwithboundary} with $R=R_n$, we see that 

\begin{align}\label{POH limit}
-\frac{c}{2}\int_{\R^3} & \phi_u u^2 (x, \nabla \rho) \dif x \nonumber\\
&= \lim_{n\to+\infty}\left(-\frac{c}{2}\int_{B_{R_n}}  \phi_u u^2 (x, \nabla \rho) \dif x\right) \nonumber \\
&= \lim_{n\to+\infty} \bigg(\frac{1}{2}\int_{B_{R_n}}|\nabla u|^2 \dif x +\frac{3b}{2}\int_{B_{R_n}}u^2\dif x +\frac{3c}{2}\int_{B_{R_n}} \rho \phi_u u^2 \dif x - \frac{c}{4}\int_{B_R}|\nabla \phi_u|^2\dif x \nonumber \\
&\qquad  -\frac{3d}{p+1}\int_{B_{R_n}}|u|^{p+1}\dif x  +\frac{1}{R_n}\int_{\partial B_{R_n}} |(x,\nabla u)|^2 \dif  \sigma -\frac{R_n}{2}\int_{\partial B_{R_n}}|\nabla u|^2 \dif \sigma\nonumber \\
&\qquad\qquad   -\frac{bR_n}{2}\int_{\partial B_{R_n}}u^2 \dif\sigma - \frac{cR_n}{2}\int_{\partial B_{R_n}}\rho \phi_u  u^2 \dif\sigma  - \frac{c}{2R_n}\int_{\partial B_{R_n}} |(x,\nabla \phi_u)|^2 \dif \sigma\nonumber\\
&\qquad\qquad\qquad  +\frac{cR_n}{4}\int_{\partial B_{R_n}}|\nabla \phi_u|^2 \dif \sigma +\frac{dR_n}{p+1}\int_{\partial B_{R_n}}|u|^{p+1}\dif \sigma \bigg) \nonumber\\
&=  \frac{1}{2}\int_{\R^3}|\nabla u|^2 \dif x +\frac{3b}{2}\int_{\R^3}u^2\dif x +\frac{3c}{2}\int_{\R^3}\rho\phi_u  u^2 \dif x  \nonumber \\
&\qquad- \frac{c}{4}\int_{\R^3}|\nabla \phi_u|^2\dif x - \frac{3d}{p+1}\int_{\R^3}|u|^{p+1}\dif x\nonumber\\
&= \frac{1}{2}\int_{\R^3}|\nabla u|^2 \dif x +\frac{3b}{2}\int_{\R^3}u^2\dif x +\frac{5c}{4}\int_{\R^3}\rho\phi_u  u^2 \dif x- \frac{3d}{p+1}\int_{\R^3}|u|^{p+1}\dif x,\\
\nonumber
\end{align}
\noindent where the final equality holds because $\int_{\R^3}|\nabla \phi_u|^2\dif x=\int_{\R^3}\rho\phi_u u^2\dif x $. Therefore, since $(u, \phi_u) \in E(\R^3) \times D^{1,2}(\R^3)$ solves \eqref{poh system}, we have shown that

\begin{equation*}
\left|\frac{c}{2}\int_{\R^3}  \phi_u u^2 (x, \nabla \rho) \dif x\right| <+\infty.\\
\end{equation*}\\
\noindent Moreover, \eqref{POH limit} also proves that

\begin{equation*}
\frac{1}{2}\int_{\R^3} |\nabla u|^2 +\frac{3b}{2}\int_{\R^3}  u^2 +\frac{5c}{4}\int_{\R^3} \rho\phi_u u^2  + \frac{c}{2}\int_{\R^3} (x,\nabla \rho) u^2\phi_u -\frac{3d}{p+1}\int_{\R^3} |u|^{p+1}=0.
\end{equation*}
\end{proof}

\renewcommand\refname{References}

\end{document}